\documentclass[preprint]{elsarticle}

\usepackage{amsmath,amssymb,latexsym, enumerate}

\newcommand{\EE}{{\mathbb E}}

\newcommand{\NN}{{\mathbb N}}
\newcommand{\PP}{{\mathbb P}}
\newcommand{\QQ}{{\mathbb Q}}
\newcommand{\RR}{{\mathbb R}}

\newcommand{\ZZ}{{\mathbb Z}}

\newcommand{\I}{\mathrm I}
\newcommand{\II}{\mathrm{II}}
\newcommand{\III}{\mathrm{III}}

\newcommand{\abs}[1]{ \left| #1 \right|}

\newcommand{\del}{\partial}

\newcommand{\ip}[1]{\langle #1 \rangle}

\newcommand{\nv}{^{-1}}

\newcommand{\nor}[2]{\left\|#1\right\|_{#2}}
\newcommand{\oline}[1]{\overline{#1}}
\newcommand{\oo}{\infty}
\newcommand{\pars}[1]{\left(#1\right)}

\newcommand{\sgn}{\mathrm{sgn}}

\DeclareMathOperator*{\esssup}{ess\,sup}
\DeclareMathOperator*{\essinf}{ess\,inf}

\newcommand{\mcl}{\mathcal}
\newcommand{\mbb}{\mathbb}
\newcommand{\mbf}{\mathbf}

\newtheorem{lemma}{Lemma}
\newtheorem{proposition}{Proposition}
\newtheorem{theorem}{Theorem}

\newdefinition{definition}{Definition}
\newproof{proof}{Proof}

\numberwithin{equation}{section}
\numberwithin{lemma}{section}
\numberwithin{proposition}{section}
\numberwithin{theorem}{section}
\numberwithin{corollary}{section}
\numberwithin{definition}{section}

\begin{document}

\title{Homogenization of pathwise Hamilton-Jacobi equations}
\author{Benjamin Seeger}

\begin{abstract}
We present qualitative and quantitative homogenization results for pathwise Hamilton-Jacobi equations with ``rough'' multiplicative driving signals.  When there is only one such signal and the Hamiltonian is convex, we show that the equation, as well as equations with smooth approximating paths, homogenize. In the multi-signal setting, we demonstrate that blow-up or homogenization may take place. The paper also includes a new well-posedness result, which gives explicit estimates for the continuity of the solution map and the equicontinuity of solutions in the spatial variable.
\end{abstract}

\maketitle

\section{Introduction} \label{S:intro}

We study the asymptotic behavior, as $\epsilon \to 0$, of pathwise Hamilton-Jacobi equations driven by a ``rough'' continuous signal $W : [0,\oo) \to \RR^M$,
\begin{equation}
		du^\epsilon + \sum_{i=1}^M H^i(Du^\epsilon, x/\epsilon) \cdot dW^i = 0  \quad \text{in } \RR^d \times (0,\oo), \qquad u^\epsilon(\cdot,0) = u_0 \quad \text{on } \RR^d. \label{E:maineq}
\end{equation}
The initial condition $u_0$ belongs to $BUC(\RR^d)$, the space of bounded, uniformly continuous functions on $\RR^d$. We expect homogenization to occur if the dependence of $H^i$ in the spatial variable $y = x/\epsilon$ has some sort of self-averaging property. In this paper, this will be periodicity or stationary ergodicity.

The interpretation of \eqref{E:maineq} is determined by the regularity of $W$. For example, if $W$ is differentiable, then $dW$ and $du^\epsilon$ represent respectively the time derivatives of $u^\epsilon$ and $W$, which we denote by $u^\epsilon_t = u^\epsilon_t(x,t)$ and $\dot W = \dot W_t = \dot W(t)$. In this case, $u^\epsilon$ is defined in the usual Crandall-Lions viscosity sense, the theory for which is outlined in the User's Guide by Crandall, Ishii, and Lions \cite{CIL}. If $W$ has bounded variation, then \eqref{E:maineq} falls within the scope of equations with $L^1$-time dependence considered by Lions and Perthame \cite{LP} and Ishii \cite{I}. In either setting, the symbol $\cdot$ stands for multiplication.

Here, we allow $W$ to be any continuous signal. The typical examples are sample paths of continuous stochastic processes, such as Brownian motion, in which case $W$ is nowhere differentiable and has unbounded variation on every interval. For such $W$, the symbol $\cdot$ should be thought of as the Stratonovich differential. 

The theory for \eqref{E:maineq} in this generality was developed by Lions and Souganidis in \cite{LS1}, \cite{LS2}, \cite{LS3}, and \cite{LS4}, and is discussed in more detail in the forthcoming book \cite{LS5}. Proving well-posedness is more challenging than in the classical viscosity setting, especially for spatially dependent $H^i$. In general, strong regularity is required for the Hamiltonians, and \eqref{E:maineq} is only well-posed for $W$ in certain H\"older spaces.

In the single path case $M = 1$, if the Hamiltonian is uniformly convex, one can weaken the regularity assumptions and prove well-posedness for any continuous $W$. This is discussed by Lions and Souganidis in a forthcoming work \cite{fLS}, and a specific example is considered by Friz, Gassiat, Lions, and Souganidis in \cite{FGLS}. 

We briefly outline the results proved in this paper, giving the precise statements later. The various assumptions, including the homogenization rate assumption \eqref{A:Hhomog} that we reference below, are listed in Section \ref{S:assumptions}.

We first study the single path setting,
\begin{equation}
		du^\epsilon + H(Du^\epsilon, x/\epsilon) \cdot dW = 0 \quad \text{in } \RR^d \times (0,\oo), \qquad u^\epsilon(\cdot,0) = u_0 \quad \text{on } \RR^d. \label{E:singlepath}
\end{equation}
Following \cite{fLS}, we prove the following new well-posedness result.

\begin{theorem}\label{IT:extension}
	Assume that $H$ is smooth, uniformly convex, and, for some $q' > 1$, positively homogenous of degree $q'$. Then, for all $\epsilon > 0$, $u_0 \in BUC(\RR^d)$, and $W \in C([0,\oo),\RR)$, \eqref{E:singlepath} admits a unique pathwise viscosity solution in the sense of Lions and Souganidis. Moreover, the solution operator for \eqref{E:singlepath} is uniformly continuous in $W$, and the modulus of continuity for $u^\epsilon(\cdot,t)$ depends only on the growth of $H$.
\end{theorem}

Using Theorem \ref{IT:extension}, we prove that, as $\epsilon \to 0$, $u^\epsilon$ converges to the unique solution of a homogenized equation of the form
\begin{equation}
		du + \oline{H}(Du) \cdot dW = 0 \quad \text{in } \RR^d \times (0,\oo), \qquad u(\cdot,0) = u_0 \quad \text{on } \RR^d. \label{E:homogeq}
\end{equation}

\begin{theorem} \label{IT:homogenization}
	In addition to the hypotheses of Theorem \ref{IT:extension}, assume \eqref{A:Hhomog}. Then there exists $\oline{H}: \RR^d \to \RR$ such that, for all $u_0 \in BUC(\RR^d)$ and $W \in C([0,\oo),\RR)$, the solution $u^\epsilon$ of \eqref{E:singlepath} converges locally uniformly, as $\epsilon \to 0$, to the solution $u$ of \eqref{E:homogeq}. Moreover, the convergence is uniform over all $u_0$ with uniformly bounded Lipschitz constant.
\end{theorem}

For more general Hamiltonians, we replace $W$ with a smooth signal $W^\epsilon$ that converges locally uniformly, as $\epsilon \to 0$, to $W$, and study the initial value problem
\begin{equation}
		u_t^\epsilon + H(Du^\epsilon, x/\epsilon)\dot{W}^\epsilon_t = 0 \quad \text{in } \RR^d \times (0,\oo), \qquad u^\epsilon(\cdot,0) = u_0 \quad \text{on } \RR^d. \label{E:singlepathmild}
\end{equation}
By imposing quantitative control on the increasing roughness of $W^\epsilon$ for small $\epsilon$, we are able to prove the following:

\begin{theorem} \label{IT:mildapproxhomog}
	Assume \eqref{A:Hhomog} and that $H$ is coercive, Lipschitz, and convex. Then there exists $\oline{H}: \RR^d \to \RR$ such that, for all $u_0 \in BUC(\RR^d)$ and $W \in C([0,\oo), \RR)$, and for a smooth approximating family $\{ W^\epsilon \}_{\epsilon > 0}$ of $W$, the solution $u^\epsilon$ of \eqref{E:singlepathmild} converges locally uniformly, as $\epsilon \to 0$, to the solution $u$ of \eqref{E:homogeq}.
\end{theorem}

We then consider the multidimensional case $M > 1$, again using a smooth approximation $W^\epsilon$ of $W$. We assume that all but one of the Hamiltonians are independent of $Du^\epsilon$, so that, after relabeling, \eqref{E:maineq} becomes, for $H: \RR^d \times \RR^d \to \RR$ and $f = (f^1,f^2,\ldots,f^{M-1}): \RR^d \to \RR^{M-1}$,
\begin{equation}
	u^\epsilon_t + H(Du^\epsilon,x/\epsilon) \dot { W}^{M, \epsilon}_t + \sum_{i=1}^{M-1} f^i(x/\epsilon) \dot W^{i,\epsilon}_t = 0  \;\; \text{in } \RR^d \times (0,\oo), \quad u^\epsilon(\cdot,0) = u_0 \;\;  \text{on } \RR^d. \label{E:maineqbadcase}
\end{equation}

\begin{theorem} \label{IT:multpath}
	If $\dot W^{M,\epsilon} \equiv 1$, $H$ grows at least linearly in $Du^\epsilon$, and $(W^1,W^2,\ldots,W^{M-1})$ has unbounded variation, then, as $\epsilon \to 0$, $\lim_{\epsilon\to 0} u^\epsilon = -\oo$. On the other hand, if $u_0 = 0$, $M=2$, and $\int_0^t \sgn(\dot W^{2,\epsilon}_s) \abs{\dot W^1_s} ds$ converges in law as $\epsilon \to 0$, then so does $u^\epsilon(x,\cdot)$ for all $x \in \RR^d$.
\end{theorem}
These contrasting examples indicate that the analysis of the more general problem \eqref{E:maineq} is sensitive to the different levels of coercivity of the Hamiltonians, as well as the correlation between the various components of the driving signal. 

The paper is organized as follows. In Section \ref{S:strategy}, we discuss the difficulties that arise in applying the classical homogenization theory to \eqref{E:maineq} and the strategies we use to overcome them. In Sections \ref{S:assumptions} and \ref{S:prelim}, we list the assumptions and some preliminary results that are used throughout the paper. In Sections \ref{S:onepath} and \ref{S:oneepspath}, we study the behavior of respectively equations \eqref{E:singlepath} and \eqref{E:singlepathmild}, and in Sections \ref{S:multiplepathblowup} and \ref{S:multiplepathconvergence}, we investigate the multiple path case \eqref{E:maineqbadcase}. Finally, in the Appendix, we state and prove the well-posedness results.

{\bf Notation.} For open $\mcl O \subset \RR^N$, $UC(\mcl O)$ ($BUC(\mcl O)$) denotes the set of (bounded) uniformly continuous functions on $\mcl O$. We write $\omega_f$ for the modulus of continuity of $f \in UC(\mcl O)$ and $\nor{f}{\mcl O} := \sup_{x \in \mcl O} |f(x)|$ for bounded $f$. We sometimes use $\nor{f}{\oo} = \nor{f}{\mcl O}$ when there is no ambiguity for $\mcl O$. If $W: [0,\oo) \to \RR$ is a continuous path, $\omega_{W,T} = \omega_W$ is the modulus of continuity of $W$ on the interval $[0,T]$, and the dependence on $T$ is suppressed where it does not cause confusion. 

We denote by $C^{0,1}(\mcl O)$ the space of bounded, Lipschitz continuous functions on $\mcl O$, and for $f \in C^{0,1}(\mcl O)$, $\nor{Df}{\oo}$ is the Lipschitz constant of $f$. For $0 < \alpha < 1$ and $0 < T \le \oo$, $C^\alpha([0,T])$ is the space of $\alpha$-H\"older continuous functions on $[0,T]$, and $[W]_{\alpha,T}$ is the $\alpha$-seminorm on $[0,T]$. Set $\NN_0 := \NN \cup \{0\}$. For $k \in \NN_0$, $C^k(\mcl O)$ ($C^k_b(\mcl O)$) are the usual spaces of functions possessing (bounded) derivatives up to order $k$. If $q \in [1,\oo]$, $q'$ denotes the dual exponent given by $\frac{1}{q} + \frac{1}{q'} = 1$, and $W^{1,q}$ and $W^{1,q}_0$ are the usual Sobolev spaces. Throughout the paper, we omit the domain or range of function spaces whenever it is clear from the context. 

The Legendre transform $G^*$ of a convex $G: \RR^d \to \RR$ is given by $G^*(p) := \sup_{z \in \RR^d} \pars{ p \cdot z - G(z) }$. For $R \ge 0$, $B_R := \left\{ x \in \RR^d : |x| \le R \right\}$ and $\Delta_R := \left\{ (x,y) \in \RR^d \times \RR^d : |x-y| \le R \right\}$. The signum of $\xi \in \RR$ is denoted by $\sgn(\xi)$. For $r \in \NN$, $I_r$ is the $r$-by-$r$ identity matrix.

For $A, B, D \in \RR$, $A \lesssim_D B$ means that there exists $C = C(D) > 0$ such that $A \le CB$. If $A \lesssim_D B$ and $A \gtrsim_D B$, then we say $A \approx_D B$. If the dependence of $C$ on $D$ is not important, we write $A \lesssim B$ or $A \approx B$.

\section{The difficulties and general strategy} \label{S:strategy} 

We first make the formal assumption that, for some $u$ solving a homogenized equation of the form $du + \oline{H}(Du,t) = 0$, the solution $u^\epsilon$ of \eqref{E:maineq} has the expansion
\[
	u^\epsilon(x,t) \approx u(x,t) + \epsilon v(x/\epsilon,t).
\]
An asymptotic analysis yields that, for fixed $p \in \RR^d$ and $t > 0$ (here, $p = Du(x,t)$ and $y = \frac{x}{\epsilon}$), $v$ should solve
\begin{equation}
	\sum_{i=1}^M H^i(D_y v + p, y) \cdot dW^i(t) = \oline{H}(p,t) \quad \text{in } \RR^d. \label{E:badcellproblem}
\end{equation}
In the classical homogenization theory, this equation, known as the cell problem or the macroscopic equation, admits appropriate solutions $v$, called correctors, for at most one value $\oline{H}(p,t)$. This uniquely determines the effective Hamiltonian, while the correctors, or certain approximations, play a central role in the proof of homogenization. See Lions, Papanicolaou, Varadhan \cite{LPV} or Evans \cite{E} for a more detailed discussion.

Due to the lack of regularity of $W$, \eqref{E:badcellproblem} has no meaning. We approximate $W$ by a $C^1$-path $W^\eta$ and consider the problem
\begin{equation}
		u^{\epsilon,\eta}_t + \sum_{i=1}^M H^i(Du^{\epsilon,\eta},x/\epsilon) \dot W^{i,\eta}_t = 0 \quad \text{in } \RR^d \times (0,\oo), \qquad u^{\epsilon,\eta}(\cdot,0) = u_0 \quad \text{on } \RR^d. \label{E:epsetaeq}
\end{equation}
We are then led to study, for fixed $p \in \RR^d$ and $\xi \in \RR^M$, the well-defined equation
\begin{equation}
	\sum_{i=1}^M H^i(D_y v + p,y)\xi^i = \oline{H}(p, \xi) \quad \text{in } \RR^d. \label{E:generalcellproblem}
\end{equation}
Formally, as $\epsilon \to 0$ for fixed $\eta$, $u^{\epsilon,\eta} \to \oline{u}^\eta$, where $\oline{u}^\eta$ solves
\begin{align}
	\oline{u}^\eta_t + \oline{H}(D \oline{u}^\eta, \dot W^\eta_t) = 0 \quad \text{in } \RR^d \times (0,\oo), \qquad \oline{u}^\eta(\cdot,0) = u_0 \quad \text{on } \RR^d. \label{E:generaletaeq}
\end{align}

The next step is to identify the limiting behavior, as $\eta \to 0$, of $\oline{u}^\eta$. This is impeded by the fact that the equation \eqref{E:generalcellproblem} for $v$ is nonlinear. In particular, we cannot expect that there exist $\oline{H}^1, \oline{H}^2, \ldots, \oline{H}^M: \RR^d \to \RR$ such that $\oline{H}$ takes the form
\begin{equation}
	\oline{H}(p,\xi) = \sum_{i=1}^M \oline{H}^i(p) \xi^i. \label{E:additivebarH}
\end{equation}
To illustrate this difficulty, we fix $p \in \RR^d$ and write $S(W^\eta)$ for the solution $\oline{u}^\eta$ of \eqref{E:generaletaeq} with initial condition $u_0(x) = p \cdot x$. If $H$ is continuous, then so is $\oline{H}$, and, hence, the solution operator $S: C^1([0,T], \RR^M) \to UC(\RR^d \times [0,T])$ is continuous in view of the comparison principle.

Now, for $\xi_1, \xi_2 \in \RR^M$, define the piecewise linear path $W^\eta$ by $W^\eta_0 = 0$ and, for $k \in \NN_0$,
\[
	\dot W^\eta_t :=
	\begin{cases}
		2\xi_1 & \text{if } t \in \pars{ 2k\eta, (2k+1) \eta }, \\
		2\xi_2 & \text{if } t \in \pars{ (2k+1)\eta, (2k+2)\eta }.
	\end{cases}
\]
As $\eta \to 0$, $W^\eta$ converges uniformly to $W_t := (\xi_1 + \xi_2)t$. Moreover, $\xi \mapsto \oline{H}(p,\xi)$ is positively $1$-homogenous, as can be seen from multiplying \eqref{E:generalcellproblem} by a positive constant, and so
\[	
	S(W^\eta)(x,t) = p \cdot x - \int_0^t \oline{H}(p, \dot W^\eta_s)\;ds
	\xrightarrow{\eta \to 0} p \cdot x - \pars{ \oline{H}(p,\xi_1) + \oline{H}(p, \xi_2)}t.
\]
On the other hand,
\[
	S(W)(x,t) = p \cdot x - \oline{H}(p,\xi_1 + \xi_2) t.
\]
Therefore, $S$ does not have a continuous extension to piecewise $C^1$-paths unless $\xi \mapsto \oline{H}(p,\xi)$ is additive. In view of the positive homogeneity, this is equivalent to \eqref{E:additivebarH}.

When $M = 1$, \eqref{E:additivebarH} holds if $\oline{(-H)} = - \oline{H}$ for the macroscopic problems
\begin{equation}
	H(D_y v_+ + p, y) = \oline{H}(p) \quad \text{and} \quad
	-H(D_y v_- + p, y) = \oline{(-H)}(p) \qquad \text{in } \RR^d. \label{E:signedcellproblem}
\end{equation}
As we demonstrate in Lemma \ref{L:convexconsistency}, this is satisfied whenever $p \mapsto H(p,y)$ is convex.

The inequality $\oline{(-H)} \ne -\oline{H}$ creates an even more striking obstruction to the convergence of $\oline{u}^\eta$. Namely, for any smooth $W^\eta$,
\[
	S(W^\eta)(x,t) = p \cdot x - \frac{ \oline{H}(p) - \oline{(-H)}(p)}{2} W^\eta_t - \frac{ \oline{H}(p) + \oline{(-H)}(p)}{2} \int_0^t \abs{ \dot W^\eta_s}\;ds.
\]
Therefore, if $W = \lim_{\eta \to 0} W^\eta$ has unbounded variation, then $\oline{u}^\eta$ may blow up in general, unless $\oline{(-H)} = - \oline{H}$.

When this equality holds, \eqref{E:generaletaeq} becomes
\begin{align}
	\oline{u}^\eta_t + \oline{H}(D \oline{u}^\eta) \dot W^\eta_t = 0 \quad \text{in } \RR^d \times (0,\oo), \qquad \oline{u}^\eta(\cdot,0) = u_0 \quad \text{on } \RR^d. \label{E:etaeq}
\end{align}
Lemma \ref{L:homogeqwellposed}, which is due to Lions and Souganidis \cite{LS2}, gives a necessary and sufficient condition for $\oline{H}$ (which is satisfied for convex $H$) that ensures that, for any initial condition $u_0$ and approximating family $\{W^\eta\}_{\eta > 0}$ of $W$, $\oline{u}^\eta$ has, as $\eta \to 0$, a local uniform limit $u$ as $\eta \to 0$. Moreover, $u$ is independent of the choice of approximating paths, and so, following \cite{LS2}, $u$ is defined to be the unique solution of \eqref{E:homogeq}.

In general, we are not able to estimate the homogenization error $u^{\epsilon,\eta} - \oline{u}^\eta$ uniformly in $\eta$, so it is necessary to estimate $u^\epsilon - u^{\epsilon,\eta}$ uniformly in $\epsilon$. This can be accomplished if $H$ satisfies \eqref{A:wellposedeq}. For more general Hamiltonians, the same strategy is used to carry out a quantitative analysis of \eqref{E:singlepathmild}.

\section{Assumptions} \label{S:assumptions}

\subsection{The Hamiltonian.} We assume that the Hamiltonians $H$ are uniformly coercive and uniformly Lipschitz continuous for bounded gradients, that is,
\begin{equation}
	\lim_{|p| \to +\oo} \inf_{y \in \RR^d} H(p,y) = +\oo, \text{ and, for all $R >0$, }H \in C^{0,1}\pars{ B_R \times \RR^d}. \label{A:Hmain}
\end{equation}
A stronger requirement is that, for some $q \ge 1$ and $0 < c \le C$,
\begin{equation}
	c(|p|^q - 1) \le H(p,y) \le C(|p|^q + 1) \quad \text{and} \quad |D_p H(p,y)| \le C (|p|^{q-1} + 1). \label{A:Hmainquant}
\end{equation}
As indicated in Section \ref{S:strategy}, we also use
\begin{equation}
	H(\cdot,y) \text{ is convex for all $y \in \RR^d$}. \label{A:Hconvex}
\end{equation}
The well-posedness results for \eqref{E:singlepath} require
\begin{equation}
	\begin{cases}
		p \mapsto H^*(p,\cdot) \text{ and } p \mapsto H(p,\cdot) \text{ are strictly convex on compact subsets of } \RR^d\backslash\{0\}, & \\[1mm]
		H^* \in C^2_b(B_R\backslash \oline{B_{1/R}} \times \RR^d) \text{ for all $R > 1$}, & \\[1mm]
		H^*(p,\cdot) > 0 \text{ for all $p \ne 0$, and} & \\[1mm]
		H^*(tp,y) = t^q H^*(p,y) \text{ for some $q>1$ and for all $(p,y) \in \RR^d \times \RR^d$ and $t > 0$.} &
	\end{cases} \label{A:wellposedeq}
\end{equation}

\subsection{The regularized paths.} For continuous $W: [0,\oo) \to \RR^M$ and $\eta > 0$, we consider the piecewise linear interpolation of $W^\eta$ of $W$ with partition size $\eta$. That is, for $k \in \NN_0$ and $t \in [k\eta,(k+1)\eta]$,
\begin{equation}
	W^\eta_t := W_{k\eta} + \frac{W_{(k+1)\eta} - W_{k\eta}}{\eta} (t - k\eta). \label{A:Wetalinear}
\end{equation}
Because $\dot W^\eta$ is not continuous, the classical viscosity theory does not immediately apply to the equation $u_t + \sum_{i=1}^M H^i(Du,x) \dot W^{i,\eta}_t = 0$. The solution is defined by solving $U_t + \sum_{i=1}^M H^i(DU,x) \xi^i = 0$ forward in time on each interval that $\dot W^\eta = \xi$ is constant.

\subsection{Periodic and random media.} A $1$-periodic medium is modeled by $a: \RR^d \to \RR$ satisfying
\begin{equation}
	 a(y + z) = a(y) \text{ for all $y \in \RR^d$ and $z \in \ZZ^d$. } \label{A:aperiodic}
\end{equation}
Periodic homogenization of Hamilton-Jacobi equations was first discussed by Lions, Papanicolaou, and Varadhan in \cite{LPV}. Their work was expanded by Evans \cite{E}, who developed the so-called perturbed test function method.

The mathematical formulation for the stationary ergodic environment is more complicated. We consider the measurable space $(\mbf \Omega, \mbf F)$ with $\mbf \Omega := C(\RR^d)$ and $\mbf F$ the $\sigma$-algebra generated by the maps $\{ a \mapsto a(y) \}_{y \in \RR^d}$. For $z \in \RR^d$, we define the translation operators $T_z: \mbf \Omega \to \mbf \Omega$ by $T_z a(y) := a(y+z)$. The family $\{T_z\}_{z \in \RR^d}$ forms a group, since, for all $z_1,z_2 \in \RR^d$, $T_{z_1} \circ T_{z_2} = T_{z_1 + z_2}$. 

We assume that there exists a probability measure $\mbf P$ on $(\mbf \Omega, \mbf F)$ such that $\{T_z\}_{z \in \RR^d}$ is measure-preserving and ergodic, that is,
\begin{equation}
	\begin{cases}
		\mbf P = \mbf P \circ T_z \text{ for all $z \in \RR^d$, and} & \\[1mm]
		\text{if } E \in \mbf F \text{ and }T_z E =  E \text{ for all } z \in \RR^d, \text{ then } \mbf P[E] = 1 \text{ or } \mbf P[E] = 0. & 
	\end{cases} \label{A:Hstatergod}
\end{equation}
Whenever we study random media, the assumptions and results are understood to hold $\mbf P$-almost surely, and, unless we specify otherwise, all constants are deterministic.

The first qualitative homogenization results for Hamilton-Jacobi equations in the random setting were obtained independently by Souganidis in \cite{S} and Rezakhanlou and Tarver in \cite{RT}. The degenerate elliptic ``viscous'' case was later treated by Lions and Souganidis in \cite{LS7}. Armstrong and Souganidis considered the same problem in \cite{AS}, introducing the so-called metric problem.

These results all used convexity, but recently, there has been some progress in the nonconvex case, for instance by Armstrong, Tran, and Yu \cite{ATY0} and \cite{ATY}, and by Armstrong and Cardaliaguet \cite{AC}. Ziliotto \cite{Z}, and later Feldman and Souganidis \cite{FS}, provided examples of Hamiltonians and random media that do not homogenize.

\subsection{Homogenization rates.} 
To keep the notation simple, it is convenient to introduce the solution operators $S^\epsilon_\pm(t): BUC(\RR^d) \to BUC(\RR^d)$ and $S_\pm(t): BUC(\RR^d) \to BUC(\RR^d)$ for respectively
\begin{align*}
	U^\epsilon_{\pm,t} \pm H(DU^\epsilon_\pm,x/\epsilon) = 0 \text{ in } \RR^d \times (0,\oo), \quad U^\epsilon_\pm(\cdot,0) = \phi \quad \text{on } \RR^d,
\end{align*}
and, for $\oline{H}, \oline{(-H)} : \RR^d \to \RR$,
\begin{align*}
	U_{\pm,t} + \oline{(\pm H)}(DU_\pm) = 0\text{ in } \RR^d \times (0,\oo), \quad U_\pm(\cdot,0) = \phi \quad \text{on } \RR^d,
\end{align*}
that is, $U^\epsilon_\pm(x,t) = S^\epsilon_\pm(t)\phi(x)$ and $U_\pm(x,t) = S_\pm(t)\phi(x)$.

We assume that there exist $\oline{H}, \oline{(-H)}: \RR^d \to \RR$ and $\beta > 0$ such that, for all $L > 0$, $\phi \in C^{0,1}(\RR^d)$ with $\nor{Du_0}{\oo} \le L$, some $\epsilon_0 = \epsilon_0(L) > 0$, and all $\tau > 0$ and $0 < \epsilon < \epsilon_0$,
\begin{equation}
	\begin{cases}
	\nor{ S^\epsilon_{\pm}(\cdot)\phi - S_{\pm}(\cdot) \phi }{R_\tau \times [0,\tau]} \lesssim_L (1+\tau)\epsilon^\beta, &  \\[1mm]
	\text{where $R_\tau = \RR^d$ or $R_\tau = B_\tau$ in respectively the periodic and random settings.}
	\end{cases}
	\label{A:Hhomog}
\end{equation}

For periodic homogenization, \eqref{A:Hmain} is enough to conclude that \eqref{A:Hhomog} holds with $\beta = \frac{1}{3}$. This was proved by Capuzzo-Dolcetta and Ishii in \cite{CDI} for the time-independent equation $u^\epsilon + H(Du^\epsilon,x,x/\epsilon) = 0$ on $\RR^d$, using a quantitative version of the perturbed test function method. A standard adaptation of the argument leads to the time dependent result. 

The rate \eqref{A:Hhomog} was obtained in the random setting by Armstrong, Cardaliaguet, and Souganidis in \cite{ACS} by quantifying the methods in \cite{AS}. The precise value of the exponent $\beta$ in this case depends on the various properties of $H$. In addition to \eqref{A:Hmain}, it is necessary to have a stronger mixing assumption than ergodicity, namely that the environment has a finite range of dependence. A full list of assumptions that give \eqref{A:Hhomog} in the random setting can be found in \cite{ACS}.

The precise quantitative statement in \cite{ACS} is that, almost surely,
\[
	\limsup_{\epsilon \to 0} \nor{S^\epsilon_{\pm}(\cdot)\phi - S_{\pm}(\cdot) \phi}{B_\tau \times [0,\tau]} \epsilon^{-\beta} \lesssim_L 1 + \tau.
\]
This is equivalent to \eqref{A:Hhomog}, but we emphasize that, in general, $\epsilon_0$ is random.

\section{Preliminary results} \label{S:prelim}

\subsection{The condition $\oline{(-H)}(p) = - \oline{H}(p)$}

As explained in Section \ref{S:strategy}, the homogenization of \eqref{E:singlepath} and \eqref{E:singlepathmild} relies on the equality
\begin{equation}
	\oline{(-H)} = - \oline{H}. \label{A:Hconsistency}
\end{equation}
It is not known how to characterize the Hamiltonians $H$ whose effective Hamiltonians satisfy \eqref{A:Hconsistency}. This is an inverse problem similar to those studied by Luo, Tran, and Yu in \cite{LTY}.

Here, we give a sufficient criterion for \eqref{A:Hconsistency}. We use the fact that $\oline{H}$ and $\oline{(-H)}$ are given by
\begin{equation}
	\oline{H}(p) = \inf_{v \in \mcl G} \sup_{y \in \RR^d} H(p + D_y v, y) \quad \text{and} \quad \oline{(-H)}(p) = \sup_{v \in \mcl G} \inf_{y \in \RR^d} \pars{ -H(p + D_y v, y) }, \label{E:barHformula}
\end{equation}
where the supremum and infimum over $y \in \RR^d$ are interpreted in the viscosity sense, and $\mcl G$ is the set of $v \in C(\RR^d)$ such that
\[
	\begin{cases}
		v \text{ is periodic} & \text{in the periodic setting, } \\[1mm]
		\lim_{|y| \to \oo} \frac{v(y)}{|y|} = 0 \text{ almost surely}& \text{in the random setting.}
	\end{cases}
\]
This is immediate for periodic $H$ in view of the comparison principle and the definition of correctors (see \cite{LPV} or \cite{E}). In the random setting, \eqref{E:barHformula} holds for convex $H$, which was proved in \cite{LS2}.

\begin{lemma} \label{L:convexconsistency}
	Assume \eqref{A:Hmain} and \eqref{A:Hconvex}. Then, in either the periodic or random setting, \eqref{A:Hconsistency} holds.
\end{lemma}

\begin{proof}
Suppose that, for some continuous $v: \RR^d \to \RR$ and $\nu \in \RR$, $H(Dv,y) \le \nu$ in $\RR^d$ in the viscosity sense. The coercivity of $H$ implies that $v$ is Lipschitz and, hence, satisfies the inequality almost everywhere. Since $H$ is convex, the converse also holds (see \cite{AS}). 

It follows that the supremum and infimum over $y \in \RR^d$ in \eqref{E:barHformula} may be interpreted in the almost everywhere sense, and, therefore,
\begin{align*}
	\oline{(-H)}(p) &= \sup_{v \in \mcl G} \essinf_{y \in \RR^d} \pars{ -H(p + D_y v, y) }
	= - \inf_{v \in \mcl G} \esssup_{y \in \RR^d} H(p + D_y v, y) 
	= -\oline{H}(p).
\end{align*}
\end{proof}

The argument above fails in the nonconvex case because viscosity inequalities are sensitive to multiplication by $-1$. Indeed, we provide a counterexample to \eqref{A:Hconsistency} in the periodic setting, based on an example from \cite{LTY}.

If $v_-$ is a corrector for $\oline{(-H)}(p)$ and $\tilde v(y) := -v_-( -y)$, then $\tilde v$ is a viscosity solution of
\[
	H(D_y \tilde v + p, -y) = - \oline{(-H)}(p).
\]
Therefore, \eqref{A:Hconsistency} is equivalent to the invariance of $\oline{H}$ under a reflection of the periodic medium, that is, $\{ (p,y) \mapsto H(p,y)\}$ and $\{ (p,y) \mapsto H(p,-y) \}$ have the same effective Hamiltonian. It follows that \eqref{A:Hconsistency} is satisfied if $H$ is even in $y$, although we will not use this here.

Following \cite{LTY}, let $F: [0,\oo) \to \RR$ be a continuous function such that
\begin{enumerate}
\item there exist $0 < \theta_3 < \theta_2 < \theta_1$ with $F(0) = 0$, $F(\theta_2) = \frac{1}{2}$, $F(\theta_3) = F(\theta_2) = \frac{1}{3}$,

\item F is strictly increasing on $[0,\theta_2]$ and $[\theta_1,+\oo)$ and strictly decreasing on $[\theta_2, \theta_1]$, and

\item $\lim_{r \to +\oo} F(r) = +\oo$,
\end{enumerate}
and, for $s \in (0,1)$, let $V_s : \RR \to \RR$ be the $1$-periodic extension of
\[
	V_s(y) :=
	\begin{cases}
		- \frac{x}{s} & \text{for } y \in [0,s],\\[1mm]
		\frac{x-1}{1-s} & \text{for } y \in [s,1].
	\end{cases}
\]
For $(p,y) \in \RR \times \RR$, set $H_s(p,y) := F(|p|) + V_s(y)$. It is shown in \cite{LTY} that $\oline{H_s} = \oline{H_{s'}}$ if and only if $s = s'$. In particular, since $H_s(p,-y) = H_{1-s}(p,y)$, $H_s$ fails to satisfy \eqref{A:Hconsistency} whenever $s \ne \frac{1}{2}$.

\subsection{Well-posedness for the homogenized equation}

Since, in general, $\oline{H}$ is only continuous, \eqref{E:homogeq} falls within the scope of the nonsmooth equations discussed in \cite{LS2} and \cite{LS5}. Here, we summarize some of the results that appear there.

We write $u = S(W,u_0)$, where the solution operator $S$ is a priori only defined for $C^1$-paths, $S: C^1([0,\oo)) \times BUC(\RR^d) \to BUC(\RR^d)$. The key assumption is that
\begin{equation}
	\oline{H} \text{ is the difference of two convex functions.} \label{A:Hdiffconvex}
\end{equation}

\begin{lemma} \label{L:homogeqwellposed} 
	The solution operator $S$ has a unique continuous extension to $C([0,\oo),\RR) \times BUC(\RR^d)$ if and only if \eqref{A:Hdiffconvex} holds. Moreover, in this case, for all $L > 0$, $u_0 \in C^{0,1}(\RR^d)$ with $\nor{Du_0}{\oo} \le L$, $W^1, W^2 \in C([0,\oo),\RR)$, and $T > 0$,
	\[
		\nor{ S(W^1,u_0) - S(W^2,u_0) }{\RR^d \times [0,T]} \lesssim_L \nor{ W^1 - W^2}{[0,T]}. 
	\]
\end{lemma}

Finding the most general condition on $H$ that implies that $\oline{H}$ satisfies \eqref{A:Hdiffconvex} is an open problem. On the other hand, the next result is standard in the homogenization literature, and follows, for instance, from \eqref{E:barHformula}.

\begin{lemma} \label{L:convexconvex}
	Assume \eqref{A:Hmain} and \eqref{A:Hconvex}. Then $\oline{H}$ is convex.
\end{lemma}

\subsection{Finite speed of propagation}

We use the classical fact that solutions of Hamilton-Jacobi equations have a cone of dependence. For a proof, we refer to Lions \cite{L}.

\begin{lemma}\label{L:finitespeed}
	Let $U^1,U^2$ be respectively Lipschitz continuous sub- and super-solutions of $v_t + H(Dv,x) = 0$ in $\RR^d \times (0,\oo)$. Define $M := \max \pars{ \nor{DU^1}{\oo} , \nor{DU^2}{\oo}}$ and $\mcl L := \nor{D_p H}{B_M \times \RR^d}$. Then, for any $R > 0$ and $0 \le t \le \frac{R}{\mathcal L}$,
	\[
		\sup_{x \in B_{R- \mathcal L t} } \pars{ U^1(x,t) - U^2(x,t)} \le \sup_{x \in B_R} \pars{ U^1(x,0) - U^2(x,0)}. 
	\]
\end{lemma}

It follows that the semigroups $S^\epsilon_\pm$ and $S_\pm$ satisfy a contraction property on balls decreasing with finite speed. This is used to track how the oscillations  of $W$ affect the homogenization error in \eqref{A:Hhomog}.

The analysis of the domain of dependence for \eqref{E:maineq} is still an open field of study. For spatially independent $H^i$ satisfying \eqref{A:Hdiffconvex}, it is shown in \cite{LS2} that the equality of the solution to a constant propagates at a finite speed. A counterexample to the finite speed of propagation for solutions of nonconvex, spatially independent $H^i$ was given by Gassiat in \cite{G}.

\section{Homogenization with a single path} \label{S:onepath}

For continuous $W: [0,\oo) \to \RR$ and $u_0 \in BUC(\RR^d)$, we prove a quantitative homogenization result for
\begin{equation}
	du^\epsilon + H(Du^\epsilon,x/\epsilon) \cdot dW = 0 \quad \text{in }\RR^d \times (0,\oo), \qquad u^\epsilon(\cdot,0) = u_0 \quad \text{on } \RR^d. \label{E:singlepathSection}
\end{equation}
We assume that $H$ satisfies \eqref{A:wellposedeq}, so that, in view of the results in the Appendix, \eqref{E:singlepathSection} admits a unique solution for any continuous $W$ and $u_0 \in BUC(\RR^d)$.

Since \eqref{A:wellposedeq} implies \eqref{A:Hmain} and \eqref{A:Hconvex}, it follows from Lemmas \ref{L:convexconsistency} and \ref{L:convexconvex} that the effective Hamiltonian $\oline{H}$ satisfies \eqref{A:Hconsistency} and \eqref{A:Hdiffconvex}. Therefore, Lemma \ref{L:homogeqwellposed} yields the well-posedness of
\begin{equation}
	du + \oline{H}(Du) \cdot dW \quad \text{in } \RR^d \times (0,\oo), \qquad u(\cdot,0) = u_0 \quad \text{on } \RR^d. \label{E:homogeqSection}
\end{equation}

To state the result, we introduce the following additional notation. Given continuous $W$ and $T > 0$, let $\chi_{W,T}: [0,\oo) \to [0,\oo)$ be the inverse of $s \mapsto s\omega_{W,T}(s)$, which is strictly monotone because the modulus $\omega_{W,T}$ is nondecreasing. In what follows, $R_T$, $\beta$, and $\epsilon_0$ are as in \eqref{A:Hhomog}.

\begin{theorem}\label{T:homogenization}
	Assume \eqref{A:wellposedeq}, \eqref{A:Hhomog}, and $u_0 \in BUC(\RR^d)$. If $u^\epsilon$ and $u$ are respectively the solutions of \eqref{E:singlepathSection} and \eqref{E:homogeqSection}, then, for all $T > 0$,
	\begin{equation*}
		\lim_{\epsilon \to 0} \nor{ u^\epsilon - u}{B_T \times [0,T]} = 0.
	\end{equation*}
	Moreover, for all $L > 0$, $u_0 \in C^{0,1}(\RR^d)$ with $\nor{Du_0}{\oo} \le L$, and $T \ge 1$, 
	\begin{equation*}
		\limsup_{\epsilon \to 0} \frac{ \nor{u^\epsilon - u}{R_T \times [0,T]} }{\psi_T\pars{\epsilon^\beta}} \lesssim_L 1,
	\text{ where }
		\psi_T(s) =
		\begin{cases}
			\omega_{W,T}\pars{ \chi_{W,T} \pars{ Ts} } & \text{in the periodic case,}\\
			\omega_{W,T}\pars{ T s^{1/2} } & \text{in the random case.}
		\end{cases}
	\end{equation*}
\end{theorem}

For example, if $W$ is $\alpha$-H\"older continuous, then the error estimate is of the form $C \epsilon^\gamma$ for some $C = C(\nor{Du_0}{\oo}, T) > 0$ and $\gamma = \frac{\alpha\beta}{\alpha+1}$ or $\gamma = \frac{\alpha \beta}{2}$ in respectively the periodic and random cases. We do not claim that these rates are sharp, especially since the existing quantitative homogenization results are not known to be sharp.

For fixed $\eta > 0$, let $W^\eta$ be the piecewise linear interpolation of $W$ with step-size $\eta$, as in \eqref{A:Wetalinear}, and let $u^{\epsilon,\eta}$ and $\oline{u}^\eta$ be the solutions of 
\begin{equation*}
	\begin{cases}
	u^{\epsilon,\eta}_t + H(Du^{\epsilon,\eta},x/\epsilon) \dot W^\eta_t = 0 \quad \text{and} \quad \oline{u}^\eta_t + \oline{H}(D\oline{u}^\eta) \dot W^\eta_t = 0 & \text{in } \RR^d \times (0,\oo), \\[1mm]
	u^{\epsilon,\eta}(\cdot,0) = \oline{u}^\eta(\cdot,0) = u_0 \quad \text{on } \RR^d.
	\end{cases}
\end{equation*}
In view of the estimates from Theorem \ref{T:extension} and Lemma \ref{L:homogeqwellposed} for respectively \eqref{E:singlepathSection} and \eqref{E:homogeqSection}, we then have the following.

\begin{lemma}\label{L:regularize}
	Assume \eqref{A:wellposedeq}. Then, for all $L > 0$, $u_0 \in C^{0,1}(\RR^d)$ with $\nor{Du_0}{\oo} \le L$, and $\epsilon,\eta> 0$,
	\[
		\nor{ u^{\epsilon,\eta} - u^\epsilon}{\RR^d \times [0,T]} + \nor{ \oline{u}^\eta - u}{\RR^d \times [0,T]} \lesssim_L \nor{W^\eta - W}{ [0,T]} \le \omega_{W,T}(\eta).
	\]
\end{lemma}

The next lemma provides an estimate for the homogenization error $u^{\epsilon,\eta} - \oline{u}^\eta$.

\begin{lemma}\label{L:homogenize}
	Assume \eqref{A:wellposedeq} and \eqref{A:Hhomog}. Then, for all $L > 0$, $u_0 \in C^{0,1}(\RR^d)$ with $\nor{Du_0}{\oo} \le L$, $0 < \eta < 1$, $0 < \epsilon < \epsilon_0$, and $T \ge 1$,
	\begin{equation*}
		\nor{u^{\epsilon,\eta} - \oline{u}^\eta}{R_T \times [0,T]} \lesssim_L \mcl E(T,\eta) \epsilon^\beta,
	\quad \text{where} \quad
		\mcl E(T,\eta) = 
		\begin{cases}
			\frac{T}{\eta} & \text{in the periodic setting,}\\[1mm]
			\frac{T^2}{\eta^2} \omega_{W,T}(\eta) & \text{in the random setting.}
		\end{cases}
	\end{equation*}
\end{lemma}

The quantity $\mcl E(T,\eta)$ arises because of the oscillations of $W^\eta$. Since $W^\eta$ is monotone on at most $N = [T/\eta]$ disjoint intervals, the error estimate \eqref{A:Hhomog} is applied $N$ times. It follows, in view of Lemma \ref{L:finitespeed}, that the radius of the spatial domain of dependence is on the order $N \omega_W(\eta)$, which is roughly the variation norm of $W^\eta$. Therefore, in general, the total error estimate is $N^2 \omega_W(\eta) \epsilon^\beta$. In the periodic case, the spatial domain of dependence is essentially the unit cell $[0,1]^d$, so the error estimate is only increased by a factor of $N$.

\newproof{proofofl}{Proof of Lemma \ref{L:homogenize}}

\begin{proofofl}
	We first observe that, for all $t > 0$,
	\begin{equation}
		\nor{Du^{\epsilon,\eta}(\cdot,t)}{\oo}, \nor{D \oline{u}^\eta(\cdot,t)}{\oo} \lesssim_L 1. \label{E:Lipestimates}
	\end{equation}
	The estimate for $Du^{\epsilon,\eta}$ follows from \eqref{E:equicontinuity}, while the one for $D \oline{u}^\eta$ is immediate from the space homogeneity and the contraction property of \eqref{E:homogeqSection}.
	
	Assume without loss of generality that, for some positive integer $N$, $T = N\eta$, and, for $i = 0,1,2,\ldots,N$, let $t_i := i \eta$. Then, if $W^\eta$ is increasing on the interval $[t_i,t_{i+1}]$,
	\begin{align*}
		u^{\epsilon,\eta}(\cdot,t_{i+1})& - \oline{u}^{\eta}(\cdot,t_{i+1})
		= S^\epsilon_+\pars{W_{t_{i+1}} - W_{t_i}}u^{\epsilon,\eta}(\cdot,t_i) - S_+\pars{W_{t_{i+1}} - W_{t_i}} \oline{u}^\eta(\cdot,t_i)
		= \I + \II,
	\end{align*}
	where
	\[
		\I := S^\epsilon_+\pars{ W_{t_{i+1}} - W_{t_i} }u^{\epsilon,\eta}(\cdot,t_i)
		- S^\epsilon_+\pars{ W_{t_{i+1}} - W_{t_i}} \oline{u}^\eta(\cdot,t_i)
	\]
	and
	\[
		\II := S^\epsilon_+\pars{ W_{t_{i+1}} - W_{t_i} } \oline{u}^\eta(\cdot,t_i)
		-S_+\pars{ W_{t_{i+1}} - W_{t_i} } \oline{u}^\eta(\cdot,t_i).
	\]
	If $W^\eta$ is decreasing on $[t_i,t_{i+1}]$, then the arguments that follow are similar, since $S^\epsilon_+$ and $S_+$ are replaced with $S^\epsilon_-$ and $S_-$, and $W_{t_{i+1}} - W_{t_i}$ is replaced with $\abs{ W_{t_{i+1}} - W_{t_i}}$.
	
	We first consider the periodic case. The contraction property for the semigroup $S^\epsilon_+$ yields
	\begin{align*}
		| \I | = \abs{ S^\epsilon_+\pars{ W_{t_{i+1}} - W_{t_i} } u^{\epsilon,\eta}(\cdot,t_i)
		- S^\epsilon_+\pars{ W_{t_{i+1}} - W_{t_i} }\oline{u}^\eta(\cdot,t_i)} \le \nor{u^{\epsilon,\eta}(\cdot,t_i) - \oline{u}^\eta(\cdot,t_i) }{\RR^d},
	\end{align*}
	and \eqref{A:Hhomog} and \eqref{E:Lipestimates} give
	\begin{align*}
		| \II |
		= \abs{ S^\epsilon_+\pars{ W_{t_{i+1}} - W_{t_i} } \oline{u}^\eta(\cdot,t_i)
		-S_+ \pars{ W_{t_{i+1}} - W_{t_i} } \oline{u}^\eta(\cdot,t_i) }
		\lesssim_L \pars{ 1 + \abs{ W_{t_{i+1}} - W_{t_i} } } \epsilon^{\beta}.
	\end{align*}
	Combining these estimates yields
	\begin{equation*}
		\nor{u^{\epsilon,\eta}(\cdot,t_{i+1}) - \oline{u}^\eta(\cdot,t_{i+1}) }{\RR^d}
		- \nor{u^{\epsilon,\eta}(\cdot,t_i) - \oline{u}^\eta(\cdot,t_i) }{\RR^d} \lesssim_L \pars{ 1 + \abs{ W_{t_{i+1}} - W_{t_i} } } \epsilon^{\beta},
	\end{equation*}
	and, hence,
	\begin{gather*}
		\nor{u^{\epsilon,\eta}(\cdot,t_N) - \oline{u}^\eta(\cdot,t_N) }{\RR^d}
		\lesssim_L \sum_{i=0}^{N-1} \pars{ 1 + \abs{ W_{t_{i+1}} - W_{t_i} } } \epsilon^{\beta}
		\lesssim_L (N + N\omega_W(\eta))\epsilon^{\beta} \\
		\lesssim_L  \pars{ \frac{T}{\eta} + \frac{T\omega_W(\eta)}{\eta}} \epsilon^{\beta}
		\lesssim_L \frac{T}{\eta} \epsilon^{\beta}.
	\end{gather*}
	For any $t \in [0,T]$, an almost identical argument gives the same bound for $\nor{ u^{\epsilon,\eta}(\cdot,t) - \oline{u}^\eta(\cdot,t)}{\RR^d}$.
	
	The one additional step in the random setting consists of controlling the propagation speed. Let $M := \max \pars{ \nor{Du^{\epsilon,\eta}}{\oo} , \nor{Du^\eta}{\oo} }$ and $\mcl L := \nor{D_p H}{B_M \times \RR^d}$, and note that, in view of \eqref{E:Lipestimates} and \eqref{A:Hbounds}, $\mathcal L \le_L 1$. Lemma \ref{L:finitespeed} yields
	\begin{align*}
		\nor{ \I}{B_R}
		\le \nor{u^{\epsilon,\eta}(\cdot,t_i) - \oline{u}^\eta(\cdot,t_i)}{B_{R + \mcl L \omega_W(\eta)} },
	\end{align*}
	while $|\II| \lesssim_L (1+R) \epsilon^\beta$ as before. Therefore,
	\begin{equation*}
		\nor{u^{\epsilon,\eta}(\cdot,t_{i+1}) - \oline{u}^\eta(\cdot,t_{i+1}) }{B_R}
		- \nor{u^{\epsilon,\eta}(\cdot,t_i) - \oline{u}^\eta(\cdot,t_i) }{B_{R+ \mcl L \omega_W(\eta)}} \lesssim_L \pars{ 1 + R } \epsilon^{\beta},
	\end{equation*}
	and we conclude that
	\begin{align*}
		\nor{u^{\epsilon,\eta}(\cdot,t_N) - \oline{u}^\eta(\cdot,t_N)}{B_T}
		\lesssim_L \sum_{i=1}^N (1+T + (i-1) \mathcal L \omega_W(\eta)) \epsilon^\beta
		\lesssim_L \frac{T^2 \omega_W(\eta)}{\eta^2} \epsilon^\beta.
	\end{align*}
\end{proofofl}

\newproof{proofoft}{Proof of Theorem \ref{T:homogenization}}

\begin{proofoft}
	We first prove the desired estimates for $u_0 \in C^{0,1}(\RR^d)$ with $\nor{Du_0}{\oo} \le L$. The homogenization for $u_0 \in BUC(\RR^d)$ then follows from a standard density argument and the contraction property of the solution operators. 
	
	In the periodic setting, for all $0 < \epsilon < \epsilon_0$ and $T,\eta > 0$,
	\begin{align*}
		\nor{u^\epsilon - u}{\RR^d \times [0,T]} \lesssim_L \omega_W(\eta) + \frac{T}{\eta} \epsilon^{\beta}.
	\end{align*}
	To optimize the above estimate, we choose $\eta$ such that $\eta\omega_W(\eta) = T \epsilon^{\beta}$.
	
	In the random setting,
	\begin{align*}
		\nor{u^\epsilon - u}{B_T \times [0,T]} 
		\lesssim_L \omega_W(\eta)\pars{ 1 + \frac{T^2}{\eta^2} \epsilon^\beta}.
	\end{align*}
	Setting $\eta = T\epsilon^{\beta/2}$ gives the optimal bound.
	\end{proofoft}

\section{Homogenization with a smooth approximating family} \label{S:oneepspath}

For some $\eta(\epsilon) > 0$ satisfying $\lim_{\epsilon \to 0} \eta(\epsilon) = 0$, we consider the piecewise linear regularization $W^{\eta(\epsilon)}$ of $W$ as in \eqref{A:Wetalinear}, and study the behavior, as $\epsilon \to 0$, of
\begin{equation}
	u^\epsilon_t + H(Du^\epsilon,x/\epsilon) \dot W^{\eta(\epsilon)}_t = 0 \quad \text{in } \RR^d \times (0,\oo), \qquad u^\epsilon(\cdot,0) = u_0 \quad \text{on } \RR^d. \label{E:singlepathmildSection}
\end{equation}

As we no longer study \eqref{E:singlepathSection} directly, it is not necessary to assume \eqref{A:wellposedeq}. Instead, we consider the weaker growth assumption \eqref{A:Hmainquant} as well as the homogenization rates \eqref{A:Hhomog}. In what follows, $R_T$, $\beta$, and $\epsilon_0$ are as in \eqref{A:Hhomog}.

We continue to assume \eqref{A:Hconvex}, although the result holds whenever $\oline{H}$ satisfies \eqref{A:Hconsistency} and \eqref{A:Hdiffconvex}.

\begin{theorem}\label{T:mildapproxhomog}
	Assume \eqref{A:Hmainquant}, \eqref{A:Hconvex}, \eqref{A:Hhomog}, $u_0 \in BUC(\RR^d)$, and
	\[
		\begin{cases}
			\lim_{\epsilon \to 0} \eta(\epsilon) = 0, & \\[1mm]
			\lim_{\epsilon \to 0} \frac{\epsilon^\beta}{\eta(\epsilon)} = 0 & \text{in the periodic case,}\\[1mm]
			\lim_{\epsilon \to 0} \frac{\abs{\log \epsilon}}{\eta(\epsilon)} = 0 & \text{in the random case.}
		\end{cases}
	\]
	If $u^\epsilon$ and $u$ are respectively the solutions of \eqref{E:singlepathmildSection} and \eqref{E:homogeqSection}, then, for all $T > 0$,
	\[
		\lim_{\epsilon \to 0} \nor{  u^\epsilon - u}{R_T \times [0,T]} = 0.
	\]
\end{theorem}

Following the strategy outlined in Section \ref{S:strategy}, we first estimate $u^{\epsilon}-\oline{u}^{\eta(\epsilon)}$, where $\oline{u}^\eta$ solves \eqref{E:etaeq}. In the periodic case, this follows exactly as in Section \ref{S:onepath}. The argument in the random setting must be adapted to account for the fact that, since \eqref{E:equicontinuity} no longer holds, $\nor{Du^\epsilon}{\oo}$ does not have a uniform bound in $\epsilon$.

\begin{lemma} \label{L:tildeuepsLip}
	Assume \eqref{A:Hmainquant}. Then there exists $c_0 \ge 0$ such that, for all $u_0 \in C^{0,1}(\RR^d)$ and $T,\epsilon > 0$, $\nor{Du^{\epsilon}(\cdot,T) }{\oo} \le \exp\pars{ c_0 T/\eta(\epsilon)} \pars{ \nor{Du_0}{\oo} + 1}$.
\end{lemma}

\begin{proof}
	The contraction property of the semigroups $S^\epsilon_\pm(t)$ and \eqref{A:Hmainquant} imply that, for any $\phi \in C^{0,1}(\RR^d)$,
	\[
		\nor{ \frac{\del}{\del t} S^\epsilon_\pm(\cdot)\phi }{\oo} \le C \pars{ \nor{D\phi}{\oo}^q + 1}.
	\]
	Therefore, for all $t > 0$, \eqref{A:Hmainquant} gives
	\[
		c \pars{ \nor{D S^\epsilon_\pm(t) \phi}{\oo}^q - 1} \le C \pars{ \nor{D\phi}{\oo}^q + 1}, 
	\]
	and, hence, for some $c_0 \ge 0$, $\nor{D S^\epsilon_\pm(t) \phi}{\oo} \le \exp(c_0) \pars{ \nor{D\phi}{\oo} + 1}$.
	
	For $k \in \NN_0$ and $t \in [k\eta, (k+1)\eta)$,
	\[
		u^{\epsilon,\eta}(\cdot,t) = S^\epsilon_{\pm}\pars{ \abs{ W^\eta_t - W^\eta_{k\eta} } }u^{\epsilon,\eta}(\cdot,k\eta),
	\] 
	depending on the monotonicity of $W^\eta$ on $[k\eta,(k+1)\eta)$. Inductively applying this formula gives the desired estimate with a possibly larger value for $c_0$.
\end{proof}

We then have the following.

\begin{lemma} \label{L:smoothpathhomog}
	Assume \eqref{A:Hmainquant} and \eqref{A:Hhomog}. Then, for all $L > 0$, $u_0 \in C^{0,1}(\RR^d)$ with $\nor{Du_0}{\oo} \le L$, $0 < \epsilon < \epsilon_0$, and $T \ge 1$,
	\begin{equation*}
		\nor{u^{\epsilon} - \oline{u}^{\eta(\epsilon)}}{R_T \times [0,T]} \lesssim_L \mcl E\pars{ \frac{T}{\eta(\epsilon)}} \epsilon^\beta,
	\end{equation*}
	where, for $M > 0$ and some universal constant $c_0 \ge 0$,
	\begin{equation*}
		\mcl E(M) = 
		\begin{cases}
			M & \text{in the periodic setting,}\\[1mm]
			\exp\pars{c_0 M} & \text{in the random setting.}
		\end{cases}
	\end{equation*}
\end{lemma}

\begin{proof} 
	The proof in the periodic case is exactly as in Lemma \ref{L:homogenize}, and therefore does not use Lemma \ref{L:tildeuepsLip}. Indeed, the growth assumption \eqref{A:Hmainquant} may be replaced with \eqref{A:Hmain}.
	
	We then turn to the random setting. Set $M := \max \pars{ \nor{Du_0}{\oo},  \nor{Du^{\epsilon,\eta}}{\oo} }$ and $\hat{\mcl L} := \nor{D_p H}{B_M \times \RR^d}$. In view of \eqref{A:Hmainquant} and Lemma \ref{L:tildeuepsLip}, for some $c_0 \ge 0$, $\hat{\mcl L} \le \exp(c_0 T \eta\nv) \pars{ \nor{Du_0}{\oo}^{q-1} + 1}$.
	
	With the notation of Lemma \ref{L:homogenize}, the same argument gives
	\[
		\nor{u^\epsilon(\cdot,t_{i+1}) - \oline{u}^{\eta(\epsilon)}(\cdot,t_{i+1})}{B_R}
		- \nor{u^\epsilon(\cdot,t_i) - \oline{u}^{\eta(\epsilon)}(\cdot,t_i)}{B_{R + \hat{\mcl L} \omega_{W,T}(\eta)}}
		\lesssim_L (1+R)\epsilon^\beta
	\]
	and so
	\[
		\nor{u^\epsilon(\cdot,t_N) - \oline{u}^{\eta(\epsilon)}(\cdot,t_N)}{B_T} \lesssim_L \sum_{k=0}^{N-1} \pars{ 1 + T + k \hat{ \mcl L} \omega_{W,T}(\eta)} \epsilon^\beta.
	\]
	We take $N = [T/\eta]$ and increase the value of $c_0$ as necessary to obtain the result.
\end{proof}

\newproof{proofoftt}{Proof of Theorem \ref{T:mildapproxhomog}}

\begin{proofoftt}
	Let $u_0^\delta$ be a standard mollification of $u_0$, so that $\nor{u_0 - u_0^\delta}{\oo} \le \omega_{u_0}(\delta)$. We combine the bounds from Lemmas \ref{L:homogeqwellposed} and \ref{L:smoothpathhomog} for the initial condition $u_0^\delta$, as well as the contraction property for \eqref{E:singlepathmildSection} and \eqref{E:homogeqSection}, to obtain
	\[
		\nor{ u^\epsilon - u}{R_T \times [0,T]} - 2 \omega_{u_0}(\delta) \lesssim_\delta
		\begin{cases}
			\frac{T}{\eta(\epsilon)} \epsilon^\beta + \omega_{W,T}(\eta(\epsilon)) & \text{in the periodic setting,}\\[1mm]
			\exp(c_0 T/\eta(\epsilon)) \epsilon^\beta + \omega_{W,T}(\eta(\epsilon)) & \text{in the random setting.}
		\end{cases}
	\]
	In either case, the assumption on $\eta(\epsilon)$ guarantees that
	\[
		\limsup_{\epsilon \to 0} \nor{ u^\epsilon- u}{R_T \times [0,T]} \le 2 \omega_{u_0}(\delta) \xrightarrow{\delta \to 0} 0.
	\]
\end{proofoftt}

We conclude with some examples in the random setting to show that it is possible to improve $\eta(\epsilon)$. 

First suppose that $c = C$ in \eqref{A:Hmainquant}. This is satisfied, for instance, by $H(p,y) = |p|^q + f(y)$ with $f \in C^{0,1}(\RR^d)$. Using this in the proof of Lemma \ref{L:tildeuepsLip}, we find, for all $0 \le t \le T$, 
\[
	\nor{ D S^\epsilon_\pm(t) \phi}{\oo} \le \nor{D\phi}{\oo} + 2 \quad \text{and} \quad \nor{Du^{\epsilon,\eta}}{\oo} \le \nor{Du_0}{\oo} + 2 \frac{T}{\eta},
\]
and, with the notation of Lemma \ref{L:smoothpathhomog}, replacing $\mcl E(M) = \mcl E(T/\eta)$ with $\mcl E(T,\eta)$, 
\[
	\hat{\mcl L} \lesssim  \nor{Du_0}{\oo}^{q-1} + \pars{\frac{T}{\eta}}^{q-1}
\quad \text{and} \quad
	\mcl E(T,\eta) = \max \pars{ \frac{T^2}{\eta^2}, \frac{T^q}{\eta^q} \omega_{W,T}(\eta) }.
\]
If $W$ is $\alpha$-H\"older continuous, for instance, it then suffices to have
\[
	\lim_{\epsilon \to 0} \frac{\epsilon^\sigma}{\eta(\epsilon)} = 0, \quad \text{where} \quad \sigma = \min \pars{ \frac{\beta}{2} , \frac{\beta}{q-\alpha}}.
\]

In the second example, we assume $q = 1$, so that the propagation speed $\hat{\mcl L}$ is uniformly bounded. Then, exactly as in Section \ref{S:onepath}, Lemma \ref{L:smoothpathhomog} yields
\[
	\mcl E(T,\eta) = \frac{T^2}{\eta^2} \omega_{W,T}(\eta).
\]
Therefore, for $\alpha$-H\"older continuous $W$, it suffices to have $\lim_{\epsilon \to 0} \epsilon^{\frac{\beta}{2-\alpha}}\eta(\epsilon)\nv = 0$.

\section{Multiple paths: the blow-up example} \label{S:multiplepathblowup}
We again take $\eta(\epsilon)$ satisfying $\lim_{\epsilon \to 0} \eta(\epsilon) = 0$, and, for $f: \RR^d \to \RR^m$, $W \in C([0,\oo), \RR^m)$, and $u_0 \in BUC(\RR^d)$, we study the behavior, as $\epsilon \to 0$, of
\begin{equation}
		u^\epsilon_t + \abs{Du^\epsilon} + \sum_{i=1}^m f^i(x/\epsilon) \dot W^{i,\eta(\epsilon)}_t = 0 \quad \text{in } \RR^d \times (0,\oo), \qquad u^\epsilon(\cdot,0) = u_0. \label{E:multdbadcase}
\end{equation}
This is a special case of \eqref{E:maineqbadcase} with $M = m+1$, $H(p,y) = |p|$, and $\dot W^M \equiv 1$.

The path $W$ is assumed to have unbounded variation on every interval and in every direction. More precisely, we assume that
\begin{equation}
	\begin{cases}
	\text{there exist $0 < c < C$ and $0 < \theta < 1$ such that, for all $T > 0$, $t \in [0,T]$, $\xi \in \RR^m$,} &\\[1mm]
	\text{and $\eta > 0$,} \quad c |\xi| \eta^{-\theta}t \le \int_0^t \abs{ \dot W^\eta_s \cdot \xi}\;ds \quad \text{and} \quad
	\int_0^t \abs{ \dot W^\eta_s}\;ds \le C\eta^{-\theta}t.&
	\end{cases} \label{A:Winfvar}
\end{equation}
When $m = 1$, a Cantor-like construction can be used to build such a path. Brownian motion also satisfies \eqref{A:Winfvar} almost surely for $0 < \theta < \frac{1}{2}$ and for some subsequence $\eta_n$ with $\lim_{n \to \oo} \eta_n = 0$.

We assume that $f$ satisfies either \eqref{A:aperiodic} or \eqref{A:Hstatergod}, as well as
\begin{equation}
	Df \in UC(\RR^d) \text{ and, for some $0 < c_0 \le C_0$, } c_0 \le \nor{Df}{\oo} \le C_0.
	\label{A:fnonconstant}
\end{equation}
In the random setting \eqref{A:Hstatergod}, we will also need to assume that the modulus of continuity for $Df$ is uniform over $f \in \mbf \Omega$.

\begin{theorem} \label{T:multdbadcase}
	In addition to \eqref{A:Winfvar}, \eqref{A:fnonconstant}, and $u_0 \in BUC(\RR^d)$, assume that $f$ satisfies one of \eqref{A:aperiodic} or \eqref{A:Hstatergod}, and, for $\theta$ as in \eqref{A:Winfvar} and $0 < \sigma < \frac{1}{1-\theta}$, set $\eta(\epsilon) = \epsilon^\sigma$. If $u^\epsilon$ solves \eqref{E:multdbadcase}, then, for all $t > 0$,
	\[
		\limsup_{\epsilon \to 0} \sup_{x \in \RR^d}  u^\epsilon(x,t) \epsilon^{\sigma \theta - (\sigma - 1)_+} \lesssim -t.
	\]
\end{theorem}

The fact that $u^\epsilon$ diverges to $-\oo$ is a consequence of the positive coercivity of $p \mapsto |p|$. If $|Du^\epsilon|$ is replaced with $-|Du^\epsilon|$ in \eqref{E:multdbadcase}, then $u^\epsilon$ diverges to $+\oo$.

\newproof{proofofttt}{Proof of Theorem \ref{T:multdbadcase}}

\begin{proofofttt}
	Without loss of generality, we take $u_0 \equiv 0$, since, if $\tilde u^\epsilon$ solves \eqref{E:multdbadcase} with $\tilde u^\epsilon(\cdot,0) \equiv 0$, then, in view of the comparison principle, $\nor{u^\epsilon - \tilde u^\epsilon}{\oo} \le \nor{u_0}{\oo}$. 
	
	We first prove that, for all $r,t> 0$,
	\begin{equation}
		\limsup_{\epsilon \to 0} \sup_{x \in B_{r \epsilon}} u^\epsilon(x,t) \epsilon^{\sigma \theta - (\sigma - 1)_+} \lesssim -t. \label{E:prelimblowup}
	\end{equation}
	Taking $r = \sqrt{d}$ gives the result in the periodic setting, since $u^\epsilon$ is $\epsilon \ZZ^d$-periodic and $B_{\sqrt{d} \epsilon} \supset \epsilon [0,1]^d$. In the random setting, the stationarity of $u^\epsilon$ and \eqref{E:prelimblowup} imply that, almost surely and for all $t > 0$,
	\[
		\limsup_{\epsilon \to 0} \sup_{x \in \QQ^d} u^\epsilon(x,t) \epsilon^{\sigma \theta - (\sigma - 1)_+} \lesssim -t. 
	\]
	The supremum may then be taken over $x \in \RR^d$ because $u^\epsilon$ is continuous.	
	
	The proof of \eqref{E:prelimblowup} relies on the variational formula
	\begin{equation}
	u^\epsilon(x,t) = - \sup \left\{  \int_0^t f \pars{ \frac{\gamma_s}{\epsilon} } \cdot \dot W^{\eta(\epsilon)}_s \;ds : \gamma \in W^{1,\oo}([0,t], \RR^d), \;\nor{\dot \gamma}{\oo} \le 1, \; \gamma_t = x \right\}. \label{E:controlformula}
\end{equation}

	Let $\omega := \omega_{Df}$ be the modulus of continuity for $Df$. It follows from \eqref{A:fnonconstant} that there exist $y_0, p \in \RR^d$ with $|p| = 1$ and $\phi: \RR^d \to \RR^m$ such that  $\xi := Df(y_0)p$ satisfies $c_0 \le |\xi| \le C_0$, $|\phi(z)| \le \omega(|z|)$ for all $z \in \RR^d$, and, for all $y \in \RR^d$,
	\[
		f(y) = f(y_0) + Df(y_0) \cdot (y-y_0) + |y-y_0| \phi(y-y_0).
	\]
	
	Set $R := \frac{r + |y_0|}{t}$ and choose $\bar\epsilon \in (0,\frac{1}{2R}]$. Note that, for all $0 < \epsilon < \bar\epsilon$, we have $N := \left[ \frac{(1 -R\epsilon)t}{\eta} \right] \in \NN$ and $N\eta \gtrsim t$.
	
	Choose $0 < \nu < \frac{1}{2}$, set $\delta := \nu \epsilon^{(\sigma-1)_+}$, and define $\alpha: [0,1] \to \RR$ by $\alpha(s) := 1 - 2|s-1/2|$. For $x \in B_{r\epsilon}$ and $k = 0, 1, 2, \ldots, N-1$, we define $\gamma: [0,t] \to \RR^d$ by
	\[
		\frac{\gamma_s}{\epsilon} := 
		\begin{cases}
			y_0 + \delta \alpha\pars{ \frac{s-k\eta}{\eta} } p & \text{if } s \in [k\eta,(k+1)\eta] \text{ and } \xi \cdot \dot W^\eta_s > 0, \\[1mm]
			y_0 - \delta \alpha\pars{ \frac{s-k\eta}{\eta} } p & \text{if } s \in [k\eta,(k+1)\eta] \text{ and } \xi \cdot \dot W^\eta_s < 0, \\[1mm]
			y_0 + \frac{s - N\eta}{t - N\eta} \pars{ \frac{x}{\epsilon} - y_0} & \text{if } x \in [N\eta,t].
		\end{cases}
	\]
	Then $\gamma$ is admissible for \eqref{E:controlformula}, since $\gamma_t = x$ and
	\[
		\nor{\dot \gamma}{\oo} \le \max \pars{ 2 \frac{\delta\epsilon}{\eta}, \frac{|x-\epsilon y_0|}{t - N\eta}} \le 1.
	\]
	
	We now calculate
	\begin{gather*}
		\int_0^{N\eta} f \pars{ \frac{\gamma_s}{\epsilon} } \cdot \dot W^\eta_s \;ds  - f(y_0) \cdot \pars{W^\eta_t - W^\eta_0}\\
		= \int_0^{N\eta} \left[ Df(y_0) \cdot \pars{ \frac{\gamma_s}{\epsilon} - y_0} + \abs{ \frac{\gamma_s}{\epsilon} - y_0} \phi\pars{ \frac{\gamma_s}{\epsilon} - y_0} \right] \cdot \dot W^\eta_s\;ds
		= \delta \sum_{k=0}^{N-1} \pars{ \I_k + \II_k},
	\end{gather*}
	where
	\begin{gather*}
		\I_k := \int_{k\eta}^{(k+1)\eta} \alpha \pars{ \frac{s - k\eta}{\eta} } \abs{ \xi \cdot \dot W^\eta_s}\;ds
		\quad \text{and} \\
		\II_k := \int_{k\eta}^{(k+1)\eta} \alpha \pars{ \frac{k-s \eta}{\eta} } \phi \pars{ \sgn\pars{ \xi \cdot \dot W^\eta_s} \delta \alpha \pars{ \frac{s-k\eta}{\eta} } p } \cdot \dot W^\eta_s \;ds.
	\end{gather*}
	
	For each $k = 0,1,2,\ldots,N-1$, the function $\xi_k := \xi \cdot \dot W^\eta$ is constant on $[k\eta, (k+1)\eta]$. Hence, \eqref{A:Winfvar} gives
	\begin{gather*}
		\sum_{k=0}^{N-1} \I_k
		=\sum_{k=0}^{N-1} \int_{k\eta}^{(k+1)\eta} \alpha \pars{ \frac{s - k\eta}{\eta} } \abs{ \xi \cdot \dot W^\eta_s}\;ds
		= \sum_{k=0}^{N-1} |\xi_k| \int_{k\eta}^{(k+1)\eta} \alpha \pars{ \frac{s-k\eta}{\eta}}\;ds \\
		= \frac{1}{2} \sum_{k=0}^{N-1} |\xi_k|\eta 
		= \frac{1}{2} \int_0^{N\eta} |\dot W^\eta_s \cdot \xi|\;ds 
		\gtrsim \eta(\epsilon)^{-\theta} t,
	\end{gather*}
	and
	\begin{gather*}
		\sum_{k=0}^{N-1} \II_k
		= \sum_{k=0}^{N-1} \int_{k\eta}^{(k+1)\eta} \alpha \pars{ \frac{s - k \eta}{\eta} } \phi \pars{ \sgn\pars{ \xi \cdot \dot W^\eta_s} \delta \alpha \pars{ \frac{s-k\eta}{\eta} } p } \cdot \dot W^\eta_s \;ds\\
		\ge - \omega(\delta) \sum_{k=0}^{N-1} \int_{k\eta}^{(k+1)\eta} \alpha \pars{ \frac{s - k\eta}{\eta} } \abs{ \dot W^\eta_s}\;ds
		\gtrsim - \omega(\delta) \eta(\epsilon)^{-\theta}t.
	\end{gather*}
	Choosing $\nu$ so that $\omega(\delta)$ is sufficiently small, we find
	\begin{gather*}
		\int_0^{N\eta} f \pars{ \frac{\gamma_s}{\epsilon} } \cdot \dot W^\eta_s \;ds  - f(y_0) \cdot \pars{W^\eta_t - W^\eta_0}
		\gtrsim \delta \eta(\epsilon)^{-\theta} t
		\gtrsim \epsilon^{-\sigma \theta + (\sigma-1)_+} t.
	\end{gather*}
	Integrating by parts on $[N\eta,t]$ gives
	\begin{gather*}
		\int_{N\eta}^{t} f \pars{ \frac{\gamma_s}{\epsilon} } \cdot \dot W^\eta_s \;ds  - f(y_0) \cdot \pars{W^\eta_t - W^\eta_{N\eta}}
		= f\pars{ \frac{x}{\epsilon} } W^\eta_t - \frac{1}{\epsilon} \int_{N\eta}^t Df \pars{ \frac{\gamma_s}{\epsilon}} \dot \gamma_s \cdot W^\eta_s\;ds,
	\end{gather*}
	and so, because $\nor{\dot \gamma}{\oo} \le 1$,
	\begin{gather*}
		\abs{\int_{N\eta}^{t} f \pars{ \frac{\gamma_s}{\epsilon} } \cdot \dot W^\eta_s \;ds  - f(y_0) \cdot \pars{W^\eta_t - W^\eta_{N\eta}}}
		\le \oline{M},
	\end{gather*}
	where $\oline M := \pars{\nor{f}{\oo} + (r+|y_0|) \nor{Df}{\oo}} \nor{W}{[0,t]}$.
	Choosing $\bar \epsilon$ sufficiently small, we conclude that, for all $0 < \epsilon < \bar \epsilon$,
	\begin{gather*}
		u^\epsilon(x,t) \le - \int_0^t f \pars{ \frac{\gamma_s}{\epsilon}} \cdot \dot W^\eta_s\;ds
		\lesssim  -\epsilon^{-\sigma \theta + (\sigma - 1)_+} t + M \lesssim -\epsilon^{-\sigma \theta + (\sigma-1)_+} t.
	\end{gather*}
\end{proofofttt}

Theorem \ref{T:multdbadcase} remains true if $H(p) = |p|$ is replaced with a Hamiltonian that grows at least linearly in $|p|$, since, in view of the comparison principle, we may reduce the problem to the study of \eqref{E:multdbadcase}.

The merit of directly studying the control formula \eqref{E:controlformula} is that we may avoid a discussion of homogenization error estimates. On the other hand, if, for each fixed $\xi \in \RR^m$, such an estimate is known for the homogenization of
\begin{equation}
	U^\epsilon_t + |DU^\epsilon| + f(x/\epsilon) \cdot \xi = 0, \label{E:xieq}
\end{equation}
then one may use the strategy of Section \ref{S:strategy} to deduce, as $\epsilon \to 0$, that $u^\epsilon \to -\oo$. This involves an analysis of the effective Hamiltonian $\oline{H}(p,\xi)$ associated to \eqref{E:xieq}.

\begin{lemma} \label{L:oscf}	
	Assume $f$ satisfies either \eqref{A:aperiodic} or \eqref{A:Hstatergod} and is not constant. Then there exist $\mu > 0$ and $b,\xi^* \in \RR^m$ with $|\xi^*| = 1$ such that, for all $(p,\xi) \in \RR^d \times \RR^m$, $\oline{H}(p,\xi) \ge b \cdot \xi + \mu |\xi \cdot \xi^*|$.
\end{lemma}

Lemma \ref{L:oscf} yields that the solution $\oline{u}^\eta$ of $\oline{u}^\eta_t + \oline{H}\pars{ D \oline{u}^\eta, \dot W^{\eta(\epsilon)}_t } = 0$ satisfies
\[
	\oline{u}^\eta(x,t) - u_0(x) + b \cdot \pars{ W^{\eta(\epsilon)}_t - W^{\eta(\epsilon)}_0} + \mu \int_0^t \abs{ W^{\eta(\epsilon)}_s \cdot \xi^*}\;ds \le 0.
\]
Therefore, as $\epsilon \to 0$, $\oline{u}^{\eta(\epsilon)} \to -\oo$. Choosing $\eta(\epsilon)$ so as to control the homogenization error estimate for $u^\epsilon - \oline{u}^\eta$, we find $u^\epsilon \to -\oo$ as $\epsilon \to 0$.

\newproof{proofofll}{Proof of Lemma \ref{L:oscf}}

\begin{proofofll}
	Let $y_1, y_2 \in \RR^d$ be such that $f(y_1) \ne f(y_2)$, and set
	\[
		b := \frac{f(y_1)+f(y_2)}{2}, \quad
		\xi^* := \frac{f(y_1) - f(y_2)}{|f(y_1) - f(y_2)|}, \quad \text{and} \quad
		\mu := \frac{|f(y_1) - f(y_2)|}{2}.
	\]
	Let $v^\gamma = v^\gamma(y; p, \xi)$ be the unique solution of the approximate corrector equation
	\[
		\gamma v^\gamma + \abs{p + D_y v^\gamma} + {f}(y) \cdot \xi = 0.
	\]
	Arguments from the classical viscosity theory give, for all $y \in \RR^d$,
	\[
		-\gamma v^\gamma(y;p, \xi) \ge b \cdot \xi + ({f}(y) - b) \cdot \xi.
	\]
	As $\gamma \to 0$, $-\gamma v^\gamma \to \oline{H}(p,\xi)$ locally uniformly in $y$, and so
	\begin{gather*}
		\oline{H}(p,\xi) \ge b \cdot \xi +  \max_{y \in \RR^d} ({f}(y) - b) \cdot \xi \ge  b \cdot \xi +  \max_{y = y_1,y_2} ({f}(y) - b) \cdot \xi 
		= b \cdot \xi + \mu |\xi \cdot \xi^*|.
	\end{gather*}
\end{proofofll}

\section{Multiple paths: the convergence example} \label{S:multiplepathconvergence}
For $H: \RR^d \times \RR^d \to \RR$, $f: \RR^d \to \RR$, a probability space $(\Omega, \mcl F, \mbb P)$, and $W^{1,\epsilon},W^{2,\epsilon}: [0,\oo) \times \Omega \to \RR$, we study the behavior of the initial value problem
\begin{equation}
		u^\epsilon_t + H(Du^\epsilon,x/\epsilon) \dot W^{1,\epsilon}_t + f(x/\epsilon) \dot W^{2,\epsilon}_t = 0 \quad \text{in }\RR^d \times (0,\oo),
		\qquad u^\epsilon(\cdot,0) = 0 \quad \text{on } \RR^d. \label{E:littleeq}
\end{equation}
The randomness generated by $(\Omega, \mcl F, \mbb P)$ is not related to the random homogenization setting discussed in Section \ref{S:assumptions}. Indeed, here we consider only periodic $H(p, \cdot)$ and $f$.

We consider the metric space $X := C([0,\oo), \RR)$ with the metric
\begin{equation}
		d(f,g) := \sum_{n=1}^\oo 2^{-n} \max \pars{ 1, \nor{f - g}{[0,n]} }.\label{E:metric}
\end{equation}
A measurable random variable $u : \Omega \to X$ gives rise to the Borel measure $u^* \PP$ on $X$ defined by $(u^* \PP)(A) := \PP\pars{ u \in A}$ for all Borel sets $A \subset X$. The object of this section is to show that, for all $x \in \RR^d$, as $\epsilon \to 0$, $\mu_x^\epsilon := u^\epsilon(x,\cdot)^* \PP$ converges weakly to the Wiener measure $\oline\mu$, under certain assumptions on $W^{1,\epsilon}$ and $W^{2,\epsilon}$. More precisely, we show that, for all $x \in \RR^d$ and bounded continuous $\phi: X \to \RR$,
\[
	\lim_{\epsilon \to 0} \int_X \phi \; d \mu_x^\epsilon = \int_X \phi \; d\oline\mu.
\]

Let $\{ X^1_k, X^2_k \}_{k =0}^\oo : \Omega \to \RR$ be measurable, independent random variables such that, for all $k \in \NN_0$,
\begin{equation}
	\begin{cases}
		\PP( X^{1}_k > 0) = \PP(X^{1}_k < 0), &\\[1mm]
		\left\{ X^{2}_k \right\}_{k=0}^\oo \text{ are identically distributed and } \EE \abs{ X^{2}_k}^2 = 1, \text{ and} & \\[1mm]
		\esssup_{\Omega} \sup_{k \in \NN_0} \abs{ \frac{ X^{2}_k}{ X^{1}_k} } < \oo. & 
	\end{cases} \label{A:randomwalks}
\end{equation}
We once again choose $\eta = \eta(\epsilon) > 0$ satisfying $\lim_{\epsilon \to 0} \eta(\epsilon) = 0$. For all $i = 1,2$, $k \in \NN_0$, and $t \in [k\eta, (k+1)\eta]$, we then set
\begin{equation}
	W^{i,\epsilon}_0 := 0 \quad \text{and} \quad W^{i,\epsilon}_t := W^{i,\epsilon}_{k\eta} + \frac{t - k\eta}{\eta^{1/2}} X^i_k. \label{A:randompaths}
\end{equation}

We also assume \eqref{A:Hmain}, and
\begin{equation}
	H \ge 0, \quad H(0,\cdot) = 0, \quad f \in BUC(\RR^d), \quad \text{and} \quad \max_{\RR^d} f = - \min_{\RR^d} f = 1. \label{A:Handf}
\end{equation}
Observe that, in view of the comparison principle, $u^\epsilon: \Omega \to BUC(\RR^d \times [0,T])$ is measurable, and, hence, the measures $\mu^\epsilon_x$ are well-defined.

\begin{theorem}\label{T:littleresult}
	Assume $W^{1,\epsilon}$ and $W^{2,\epsilon}$ are given by \eqref{A:randompaths}, $H$ and $f$ satisfy \eqref{A:Hmain}, \eqref{A:aperiodic} and \eqref{A:Handf}, and $\lim_{\epsilon \to 0} \frac{\epsilon}{\eta(\epsilon)} = 0$. Then, for all $x \in \RR^d$, as $\epsilon \to 0$, $\mu^\epsilon_x$ converges weakly to $\oline{\mu}$.
\end{theorem}

The result depends on an explicit formula, for fixed $\xi_1,\xi_2 \in \RR$, for the effective constant $\lambda = \lambda(\xi,\xi_2)$ of
\begin{equation}
	H(D_y v, y)\xi_1 + f(y) \xi_2 = \lambda. \label{E:littlecorrectoreq}
\end{equation}
That is, we determine the unique constant $\lambda$ for which \eqref{E:littlecorrectoreq} admits periodic solutions.

\begin{lemma}\label{L:Hbarandcorrectors}
	Assume $H$ and $f$ satisfy \eqref{A:Hmain}, \eqref{A:aperiodic}, and \eqref{A:Handf}. Then, for any $\xi_1, \xi_2 \in \RR$ with $\xi_1 \ne 0$, there exists a periodic solution $v = v(y; \xi_1,\xi_2)$ of \eqref{E:littlecorrectoreq} with $\lambda = \sgn(\xi_1) |\xi_2|$ satisfying $\nor{v}{\oo} + \nor{D_y v}{\oo} \lesssim_{\abs{\xi_2/\xi_1}} 1$.
\end{lemma}

\begin{proof}
	The existence, for some unique $\lambda$, of a periodic $v$ solving \eqref{E:littlecorrectoreq} and satisfying the desired bounds is standard (see, for example, \cite{LPV} or \cite{E}). Here, we only prove that $\lambda = \sgn(\xi_1) |\xi_2|$.
	
	 Assume first that $\xi_1 > 0$ and note that, without loss of generality, we may take $\xi_1 = 1$. It follows from \eqref{A:Handf} that $\lambda \ge \max_{y \in \RR^d} ( f(y) \xi_2 ) = | \xi_2|$. Since $v$ is periodic, $v$ attains a minimum at some $y_0 \in \RR^d$. This gives $0 + f(y_0) \xi_2 \ge \lambda$, and, therefore, $\lambda = |\xi_2|$. A similar argument yields $\lambda = -|\xi_2|$ if $\xi_1 < 0$.
\end{proof}

The bounds for the corrector $v$ in Lemma \ref{L:Hbarandcorrectors} and an argument as in \cite{LPV} result in a rate for the homogenization, as $\epsilon \to 0$, of
\begin{equation}
	U^\epsilon_t + H(DU^\epsilon,x/\epsilon)\xi_1 + f(x/\epsilon)\xi_2 = 0 \quad \text{in } \RR^d \times (0,\oo), \qquad U^\epsilon(\cdot,0) = 0 \quad \text{on } \RR^d. \label{E:littlexieq}
\end{equation}

\begin{lemma} \label{L:localtimeconv}
	Assume $H$ and $f$ satisfy \eqref{A:aperiodic} and \eqref{A:Handf}. Then, for all $\xi_1, \xi_2 \in \RR$ with $\xi_1 \ne 0$, 
	\[
		\sup_{(x,t) \in \RR^d \times [0,\oo)} \abs{ U^\epsilon(x,t) + \sgn(\xi_1)|\xi_2| \cdot t} \lesssim_{\abs{\xi_2/\xi_1}} \epsilon.
	\]
\end{lemma}

In view of Lemmas \ref{L:Hbarandcorrectors} and \ref{L:localtimeconv}, if $W^{1,\epsilon}$ and $W^{2,\epsilon}$ are viewed as fixed paths independent of $\epsilon$, then the solution $u^\epsilon$ of \eqref{E:littleeq} formally homogenizes to
\[
	W^\epsilon_t := - \int_0^t \sgn\pars{\dot W^{1,\epsilon}_s} \abs{\dot W^{2,\epsilon}_s}\;ds.
\]
Here, $W^\epsilon$ takes the role of the solution $\oline{u}^\eta$ of \eqref{E:generaletaeq} in Section \ref{S:strategy}.

Observe that $W^\epsilon$ is piecewise linear with step size $\eta(\epsilon)$. Moreover, \eqref{A:randomwalks} yields that, for $k \in \NN_0$, the random variables
\[
	\delta W^\epsilon_k := W^\epsilon_{(k+1)\eta} - W^\epsilon_{k\eta}
	= - \eta(\epsilon)^{1/2} \sgn(X^1_k) \abs{X^2_k}
\]
are independent, identically distributed, and satisfy
\[
	\EE\left[ \delta W^{\epsilon}_k \right] = 0
	\quad \text{and} \quad
	\EE \abs{ \delta W^\epsilon_k }^2 = \eta(\epsilon).
\]
Donsker's invariance principle (see Billingsley \cite{B}) implies that, as $\epsilon \to 0$, $W^\epsilon$ converges in law to a Brownian motion, that is, $(W^\epsilon)^*\PP$ converges weakly to the Wiener measure $\oline{\mu}$. In order to obtain the same conclusion for $\mu^\epsilon_x$ and prove Theorem \ref{T:littleresult}, it is necessary to connect the behavior of $u^\epsilon$ to $W^\epsilon$.

An iterative application of Lemma \ref{L:localtimeconv} yields the following estimates for $u^\epsilon-W^\epsilon$. The proof is similar to those for Lemmas \ref{L:homogenize} and \ref{L:smoothpathhomog}.

\begin{lemma} \label{L:globaltimeconv}
	Assume $W^{1,\epsilon}$ and $W^{2,\epsilon}$ are given by \eqref{A:randompaths}, and $H(p,\cdot)$ and $f$ satisfy \eqref{A:aperiodic} and \eqref{A:Handf}. Then, almost surely, for all $T \ge 1$ and $\epsilon > 0$,
	\[
		\nor{ u^\epsilon - W^\epsilon }{\RR^d \times [0,T]} \lesssim \frac{T}{\eta(\epsilon)}\epsilon.
	\]
\end{lemma}

\begin{proof} For $\xi = (\xi_1,\xi_2) \in \RR^2$ and $t > 0$, let $S^\epsilon_\xi(t): BUC(\RR^d) \to BUC(\RR^d)$ be the solution operator for \eqref{E:littlexieq}, and for $k \in \NN_0$, set $\xi_1 = \eta(\epsilon)^{-1/2} X^1_k$ and $\xi_2 = \eta(\epsilon)^{-1/2} X^2_k$. Then, for all $t \in [k\eta,(k+1)\eta)$,
	\[
		u^\epsilon(\cdot,t) = S^\epsilon_\xi(t - k\eta)u^\epsilon(\cdot,k\eta) \quad \text{and} \quad W^\epsilon_t = W^\epsilon_{k\eta} - \sgn(\xi_1) |\xi_2|(t - k\eta).
	\]
	We then have $u^\epsilon(\cdot,t) - W^\epsilon_t = \I + \II$, where
	\[
		\I := S^\epsilon_\xi(t - k\eta)u^\epsilon(\cdot,k\eta) - S^\epsilon_\xi(t - k\eta) W^\epsilon_{k\eta}
	\]
	and
	\[
		\II := S^\epsilon_\xi(t - k\eta)(W^\epsilon_{k\eta}) - W^\epsilon_{k\eta} + \sgn(\xi_1)|\xi_2|(t - k\eta).
	\]
	The contraction property for $S^\epsilon_\xi$ gives $|\I| \le \nor{u^\epsilon(\cdot,k\eta) - W^\epsilon_{k\eta}}{\RR^d}$. Meanwhile, since $S^\epsilon_\xi$ commutes with constants, Lemma \ref{L:localtimeconv} and the third line of \eqref{A:randomwalks} almost surely yield
	\[
		|\II| \le \abs{ S^\epsilon_\xi(t-k\eta)(0) - \sgn(\xi_1)|\xi_2|(t - k\eta) } \lesssim_{\abs{ \xi_2/\xi_1}} \epsilon \lesssim \epsilon.
	\]
	An iteration gives the result exactly as in the proof of Lemma \ref{L:homogenize}.
\end{proof}

Finally, we need the following.

\begin{lemma} \label{L:weakconv}
	Let $w^{1,\epsilon}, w^{2,\epsilon} : \Omega \to X$ be measurable, where $(\Omega, \mcl F, \PP)$ and $(X,d)$ are respectively a probability and metric space, and, for $i = 1,2$, set $\mu^{i,\epsilon} := (w^{i,\epsilon})^* \PP$. Assume that, for some deterministic $c(\epsilon)$ satisfying $\lim_{\epsilon \to 0} c(\epsilon) = 0$, $\PP\left[ d(w^{1,\epsilon}, w^{2,\epsilon}) \le c(\epsilon)\right] = 1$. Then, as $\epsilon \to 0$, either both $\mu^{1,\epsilon}$ and $\mu^{2,\epsilon}$ converge weakly to the same Borel measure on $X$ or neither converges.
\end{lemma}

\begin{proof}
	A necessary and sufficient condition for a measure $\mu_\epsilon$ to converge weakly, as $\epsilon \to 0$, to some measure $\oline{\mu}$ is that $\liminf_{\epsilon \to 0} \mu_\epsilon(G) \ge \oline{\mu}(G)$ for every open $G \subset X$ (see \cite{B}).
	
	For $\delta > 0$, define $G_\delta := \{ x \in X : B_\delta(x) \subset G\}$ and $\oline{c}(s) := \sup_{0 < r \le s} c(r)$. Then, for every $0 < \epsilon < \epsilon'$, $\mu^{2,\epsilon}(G) \ge \mu^{1,\epsilon}(G_{\oline{c}(\epsilon')})$. Sending first $\epsilon \to 0$ and then $\epsilon' \to 0$, and switching the roles of $\mu^{1,\epsilon}$ and $\mu^{2,\epsilon}$, we obtain $\liminf_{\epsilon \to 0} \mu^{1,\epsilon}(G) = \liminf_{\epsilon \to 0} \mu^{2,\epsilon}(G)$, and the result follows.
\end{proof}

\newproof{proofoftttt}{Proof of Theorem \ref{T:littleresult}}

\begin{proofoftttt}
	We apply Lemma \ref{L:weakconv} to $X = C([0,\oo), \RR)$, the metric $d$ in \eqref{E:metric}, the paths $\{t \mapsto u^\epsilon(x,t)\},\{ t \mapsto W^\epsilon_t\} \in X$, and, for some $C > 0$ determined by Lemma \ref{L:globaltimeconv},
	 \[
	 	c(\epsilon) := C \frac{\epsilon}{\eta(\epsilon)}.
	 \]
	 We obtain the result in view of Lemmas \ref{L:globaltimeconv} and \ref{L:weakconv}, and from the fact that, as $\epsilon \to 0$, $W^\epsilon$ converges in law to a Brownian motion.
\end{proofoftttt}

Let $W^1$ and $W^2$ be as in \eqref{A:randompaths} with $\eta = 1$. It is natural to ask whether Theorem \ref{T:littleresult} can be used to describe the long-time behavior of 
\begin{equation}
	u_t + H(Du,y)\dot W^1_t + f(y) \dot W^2_t = 0 \quad \text{in } \RR^d \times (0,\oo), \qquad u(\cdot,0) = 0 \quad \text{on } \RR^d. \label{E:longtime}
\end{equation}
Indeed, the scaling $u^\epsilon(x,t) := \epsilon u\pars{ \frac{x}{\epsilon}, \frac{t}{\epsilon^2} }$ recovers the solution of \eqref{E:littleeq} with $\eta(\epsilon) = \epsilon^2$. Unfortunately, this does not satisfy $\lim_{\epsilon \to 0} \epsilon \eta(\epsilon)^{-1} = 0$, and so we cannot directly apply Theorem \ref{T:littleresult}. We hope to study the long-time behavior of \eqref{E:longtime} and similar equations at a future time.

We conclude with the observation that Lemma \ref{L:localtimeconv} can be used to obtain uniform convergence results for equations like \eqref{E:littleeq} that are not covered by Theorem \ref{T:mildapproxhomog}. For instance, if, for some $\mu_\eta > 0$, $W^{1,\epsilon}_0 = W^{2,\epsilon}_0 = 0$ and, for $k \in \NN_0$,
\[
	\dot W^{1,\epsilon}_t :=
	\begin{cases}
		\mu_\eta & \text{if }t \in (4k\eta, (4k+2)\eta), \\
		-\mu_\eta & \text{if }t \in ((4k+2)\eta, (4k+4)\eta),
	\end{cases}
\quad \text{and} \quad
	\dot W^{2,\epsilon}_t :=
	\begin{cases}
		\mu_\eta & \text{if } t \in (2k\eta, (2k+1)\eta), \\
		-\mu_\eta & \text{if } t \in ((2k+1)\eta, (2k+2)\eta),
	\end{cases}
\]
then, for all $T >0$,
\[
	\nor{u^\epsilon + W^{1,\epsilon} }{\RR^d \times [0,T]} \lesssim \frac{T}{\eta} \epsilon.
\]
Hence, if $\lim_{\eta \to 0} \mu_\eta \eta = \lim_{\epsilon \to 0} \frac{\epsilon}{\eta(\epsilon)} = 0$, then, as $\epsilon \to 0$, $u^\epsilon$ converges uniformly to $0$. 

\appendix

\section{Well-posedness} \label{S:wellposedness}

\renewcommand{\thetheorem}{\Alph{section}.\arabic{theorem}}
\renewcommand{\thelemma}{\Alph{section}.\arabic{lemma}}
\renewcommand{\thedefinition}{\Alph{section}.\arabic{definition}}
\renewcommand{\theproposition}{\Alph{section}.\arabic{proposition}}
\renewcommand{\thesubsection}{\Alph{section}.\arabic{subsection}}

For $W \in C([0,\oo), \RR)$ and $u_0 \in BUC(\RR^d)$, we study the well-posedness of the initial value problem
\begin{equation}
	du = H(Du,x) \cdot dW \quad \text{in } \RR^d \times (0,\oo), \qquad u(\cdot,0) = u_0 \quad \text{on } \RR^d. \label{E:appeq}
\end{equation}
Throughout this section, we assume \eqref{A:wellposedeq}. It is straightforward to check that \eqref{A:wellposedeq} yields, for some $0 < c_0 \le C_0$ and $C > 0$ and for all $(p,y) \in \RR^d \times \RR^d$, 
\begin{equation}  
	\begin{cases}
		c_0|p|^q \le H^*(p,y) \le C_0|p|^q,  &|D_y H^*(p,y)| \le C|p|^q,\\[1mm]
		c_0|p|^{q-1} \le |D_p H^*(p,y)| \le C_0|p|^{q-1},  &|D^2_{py} H^*(p,y)| \le C|p|^{q-1},\\[1mm]
		c_0|p|^{q-2}\; I_d \le D^2_p H^*(p,y) \le C_0|p|^{q-2}\; I_d,   &|D^2_{y} H^*(p,y)| \le C|p|^q.
	\end{cases}
	\label{A:Hbounds}
\end{equation}
Standard convex analysis implies that \eqref{A:wellposedeq} holds with $H$ and $H^*$ interchanged, and with $q' = \frac{q}{q-1}$ replacing $q$. Hence, $H$ satisfies \eqref{A:Hbounds} with the exponent $q'$.

The definition of the Lions-Souganidis pathwise viscosity solutions, as it appears in \cite{LS1}, relies on the existence, for all $t_0 > 0$, $\phi \in C^2_b(\RR^d)$, and sufficiently small $h > 0$, of smooth-in-$x$ solutions of
\begin{equation}
	d\Phi = H(D\Phi,x) \cdot dW \quad \text{in } \RR^d \times (t_0 - h,t_0 + h), \qquad \Phi(\cdot,t_0) = \phi \quad \text{on } \RR^d.
\label{E:smoothtestfunc}
\end{equation}
These are used as test functions to account for the rough part of the equation. The solution $\Phi$ may be obtained via a change of variables from a smooth solution of the classical equation $U_t = H(DU,x)$, which can be constructed using the method of characteristics.

\begin{definition}\label{D:solution}
	The upper (resp. lower) semicontinuous function $u : \RR^d \times [0,\oo) \to \RR$ is a sub-solution (resp. super-solution) of \eqref{E:appeq} whenever the following holds: for any $\psi \in C^1((0,\oo),\RR)$, $h > 0$, and smooth-in-$x$ solution $\Phi$ of \eqref{E:smoothtestfunc} in $\RR^d \times (t_0 - h,t_0+h)$, if $u(x,t) - \Phi(x,t) - \psi(t)$ attains a local maximum (resp. minimum) at $(x_0,t_0)$, then $\psi'(t_0) \le 0$ (resp. $\psi'(t_0) \ge 0$). A solution of \eqref{E:appeq} is both a sub-solution and a super-solution. 
\end{definition}

Before we state the well-posedness theorem, we introduce some additional notation. For any modulus of continuity $\omega: [0,\oo) \to [0,\oo)$,
\[
	\theta(\omega,\lambda) := \sup_{r \ge 0} \pars{ \omega(r) - c_0 \lambda^{q-1} r^q }
	\quad
	\text{and}
	\quad
	\tilde \omega(s) := \inf_{\lambda \ge 0} \pars{ C_0 \lambda^{q-1} s^q + \theta(\omega_{u_0}, \lambda) }.
\]
Observe that $\lim_{\lambda \to \oo} \theta(\omega,\lambda) = \lim_{s \to 0} \tilde \omega(s) = 0$. 

A particular example used throughout the paper is the Lipschitz modulus $\omega(r) = Lr$. For some $C_1,C_2 > 0$,
\[
	\theta(\omega,\lambda) = C_1 L^{q'} \lambda\nv \quad \text{and} \quad \tilde \omega(s) = C_2 L s.
\] 

\begin{theorem} \label{T:extension}
	Assume \eqref{A:wellposedeq}. Then, for every $u_0 \in BUC(\RR^d)$, there exists a unique pathwise viscosity solution $u$ of \eqref{E:appeq}. Moreover, if $u^1$ and $u^2$ are the unique solutions of \eqref{E:appeq} with respectively initial conditions $u_0^1, u_0^2 \in BUC(\RR^d)$ and paths $W^1, W^2 \in C([0,\oo), \RR)$ with $W^1_0 = W^2_0$, then, for all $t > 0$,
	\begin{equation}
		\nor{ u^1(\cdot,t) - u^2(\cdot,t)}{\oo} \le \nor{u_0^1 - u_0^2}{\oo} + \theta\pars{ \max(\omega_{u_0}, \omega_{v_0}), \frac{1}{\nor{W^1- W^2}{[0,t]} } }\label{E:extension}
	\end{equation}
	and
	\begin{equation}
		\sup_{x,y \in \RR^d} \abs{ u^1(x,t) - u^1(y,t)} \le \widetilde{\omega_{u_0^1}}(|x-y|). \label{E:equicontinuity}
	\end{equation}
\end{theorem}

A standard density argument and \eqref{E:extension} yield the existence, since Definition \ref{D:solution} is consistent with the classical notion of solution for \eqref{E:appeq} if $W$ is smooth. Uniqueness, \eqref{E:extension}, and \eqref{E:equicontinuity} follow from comparison principle arguments.

If $u$ and $v$ are respectively a sub- and super-solution of \eqref{E:appeq}, then $z(x,y,t) := u(x,t) - v(y,t)$ is a sub-solution of the doubled equation
\begin{equation}
	dz = \pars{ H(D_x z, x) - H(-D_y z,y) } \cdot dW \quad \text{in } \RR^d \times \RR^d \times (0,\oo). \label{E:doubledeq}
\end{equation}
The proof of the comparison principle in the classical viscosity theory is based on finding an estimate, as $\lambda \to \oo$, for
\begin{equation}
	\sup_{(x,y) \in \RR^d \times \RR^d} \pars{ u(x,t) - v(y,t) - (\lambda/2)|x-y|^2}. \label{E:difference}
\end{equation}
If $H(p,x) = H(p)$ is independent of $x$, then $L_\lambda(x,y) := (\lambda/2)|x-y|^2$ is a smooth solution of
\begin{equation}
	H(D_x L_\lambda,x) = H(-D_y L_\lambda,y), \label{E:cancellation}
\end{equation}
so, in view of Definition \ref{D:solution}, \eqref{E:difference} is nondecreasing in $t$. However, when $H$ depends on $x$, we expect error terms like $|dW|$ to appear, which in general may be infinite. One strategy is to replace $(\lambda/2)|x-y|^2$ with a smooth solution of \eqref{E:doubledeq} that is equal to $(\lambda/2)|x-y|^2$ for $t$ near the maximum point of \eqref{E:difference}. Strong regularity assumptions for both $H$ and $W$ are required to carry out this analysis, and the details may be found in \cite{LS5}. The idea of \cite{FGLS} and \cite{fLS} is that a convex Hamiltonian $H$ gives rise to a distance function, that is, an exact solution $L_\lambda$ of \eqref{E:cancellation} that is comparable to $|x-y|$.

\subsection{The distance function}

For $x,y \in \RR^d$, define
\[
	\mcl A(x,y) := \left\{ \gamma \in W^{1,\oo}( [0,1], \RR^d) : \gamma_0 = x, \; \gamma_1 = y \right\} \quad \text{and} \quad \ell(x,y)_s = \ell_s := x + s(y-x).
\]
Note that $\ell(x,y) \in \mcl A(x,y)$ and $\mcl A(x,y) = W^{1,\oo}_0([0,1],\RR^d) + \ell(x,y)$.

The distance function associated to $H$ is 
\begin{equation}
	L(x,y) := \inf \left\{ \int_0^1 H^*\pars{ - \dot \gamma_s, {\gamma_s}}\;ds : \gamma \in \mcl A(x,y) \right\}. \label{E:distancefn}
\end{equation}
We summarize its main properties in the next lemma.

{
\setlength{\parskip}{0in}

\begin{lemma}\label{L:distancefunction} 
Assume \eqref{A:wellposedeq}.
\begin{enumerate}[(a)]
\item \label{L:Lequations} $L$ is a viscosity solution of
	\begin{equation*}
		-(q-1)L + H(D_x L,x) = 0 \quad \text{and} \quad -(q-1)L + H(-D_y L,y) = 0 \quad \text{in } \RR^d \times \RR^d.
	\end{equation*}
	In particular, $H(D_x L(x,y),x) = H(-D_y L(x,y),y)$ whenever $L$ is differentiable at $(x,y)$.
\item \label{L:Lcoercive} For all $x,y \in \RR^d$, $c_0 |x-y|^q \le L(x,y) \le C_0 |x-y|^q$. Furthermore, there exists $\gamma \in \mcl A(x,y)$ such that $L(x,y) = \int_0^1 H^*(- \dot \gamma_s, \gamma_s)\;ds$, and, for almost every $s \in [0,1]$, $|\dot \gamma_s| \approx |x-y|$.
\item \label{L:LLipschitz} For all $R > 0$, $|D_x L| + |D_y L| \lesssim_R 1$ on $\Delta_R$.
\item \label{L:Lsmoothstrip} There exists $r_0 > 0$ such that $L \in C^1(\Delta_{r_0})$.
\item \label{L:Lsemiconcave} For all $R > 1$, $D^2 L(x,y) \lesssim_R I_{2d}$ on $\Delta_R$ if $q \ge 2$ or on $\Delta_R \backslash \oline{\Delta_{1/R}}$ if $q < 2$.
\end{enumerate}
\end{lemma}
}

If $g$ is a smooth Riemannian metric on $\RR^d$ and $H^*(p,y) = \frac{1}{2} \ip{ g(y)p, p}$, then the distance function $L$ corresponding to $H(p,y) = \frac{1}{2} \ip{ g\nv(y)p, p}$ is the so-called Riemannian energy associated to $g$. In \cite{FGLS}, it is shown that $L$ is $C^1$ in a neighborhood of the diagonal $\Delta_0$, which is a specific case of Lemma \ref{L:distancefunction}\eqref{L:Lsmoothstrip}. The argument for part \eqref{L:Lsmoothstrip} resembles that of the more general setting in \cite{fLS}.

For $\epsilon > 0$, the distance function $L^\epsilon$ associated to $H(\cdot,\cdot/\epsilon)$ is
\[
	L^\epsilon(x,y) := \inf \left\{ \int_0^1 H^*\pars{ - \dot \gamma_s, \frac{\gamma_s}{\epsilon}}\;ds : \gamma \in \mcl A(x,y) \right\}. 
\]
A rescaling and part \eqref{L:Lsmoothstrip} yield $L^\epsilon \in C^1(\Delta_{\epsilon r_0})$. Therefore, the strip in which $L^\epsilon$ is differentiable shrinks as $\epsilon \to 0$. This presents a obstacle for obtaining the scale invariant estimate \eqref{E:extension}. 

We bypass this difficulty with the semiconcavity estimate for $L$. If $\phi(x,y)$ is smooth and $L - \phi$ attains a minimum at $(x_0,y_0)$, then, in view of Lemma \ref{L:distancefunction}\eqref{L:Lsmoothstrip} and \eqref{L:Lsemiconcave}, $L$ is differentiable at $(x_0,y_0)$.

\newproof{proofofal}{Proof of Lemma \ref{L:distancefunction}}
\begin{proofofal}

\eqref{L:Lequations} This follows from well known variational formulae for solutions of Hamilton-Jacobi equations. See, for instance, Lions \cite{Lbook}.

\eqref{L:Lcoercive} The bounds for $L$ are immediate from \eqref{A:Hbounds} and \eqref{E:distancefn}. In view of the convexity and regularity of $H^*$, a classical variational argument yields the existence of a minimizer $\gamma$. The bounds for $\dot \gamma$ can then be inferred from the Euler-Lagrange equation.

\eqref{L:LLipschitz} Pick $(x,y) \in \Delta_R$ and $h \in \RR^d$, and let $\gamma \in \mcl A(x,y)$ be a minimizer for $L(x,y)$. Then $\left\{ s \mapsto \gamma_s + s h : s \in [0,1] \right\} \in \mcl A(x,y+h)$. It follows from part \eqref{L:Lcoercive} that
\begin{gather*}
	L(x, y+h) - L(x,y) 
	\le \int_0^1 \pars{ H^*\pars{-\dot \gamma_s - h, \gamma_s + sh} - H^*(-\dot \gamma_s, \gamma_s)}\;ds \\
	\lesssim \int_0^1 \left[ \pars{ |\dot \gamma_s|^{q-1} +|h|^{q-1} } \abs{h} + \pars{ |\dot \gamma_s|^q + \abs{h}^q} \abs{s h} \right]\;ds
	\lesssim |h| (R^q + R^{q-1}) + o(|h|).
\end{gather*}
The opposite inequality is obtained by switching the roles of $y$ and $y+h$, and choosing $h$ small enough that $(x,y+h) \in \Delta_R$. This yields the bound for $D_y L$, and the argument for $D_x L$ is similar.

\eqref{L:Lsmoothstrip} Define $I: W_0^{1,q}([0,1],\RR^d) \times \RR^d \times \RR^d \to \RR$ by 
\[
	I(\gamma,x,y) := \int_0^1 H^*(-\dot \gamma_s + x-y, \gamma_s + x + s(y-x))\;ds.
\]
Henceforth, the arguments of $H^*$ and all of its derivatives are $(-\dot { \gamma}_s + x-y,  \gamma_s + x + s(y-x))$.
	
Observe that $L(x,y) = \min_{\gamma \in W_0^{1,q}} I(\gamma,x,y)$, and, in view of part \eqref{L:Lcoercive}, the minimum is attained for some $\gamma \in W^{1,\oo}_0$ satisfying $|\dot \gamma_s + y - x| \approx |x-y|$ for almost every $s$. As a result, we can assume that $H^*$ grows at most quadratically by redefining $H^*$ outside of $B_{R} \times \RR^d$ for some large $R > 0$. It follows that the map $I$ can be defined as before on $W_0^{1,2} \times \RR^d \times \RR^d$, and that $I$ has the same minimizers. 
	
For $x,y \in \RR^d$, fix a minimizer $\gamma$. Then $D_\gamma I(\gamma, x,y) = 0$, and, for all $\xi, \eta \in W_0^{1,2}$, 
\begin{align*}
	D^2_\gamma I(\gamma,x,y)[\xi,\eta]
	= \int_0^1 \pars{ \ip{ D^2_p H^* \dot \xi_s, \dot \eta_s}
	- \ip{ D^2_{px} H^*\dot \xi_s, \eta_s}
	- \ip{D^2_{xp} H^*\xi_s, \dot \eta_s}
	+ \ip{D^2_x H^*\xi_s, \eta_s} }\;ds.
\end{align*}
In view of \eqref{A:Hbounds}, there exists $C > 0$ such that
\begin{align*}
	D^2_\gamma I(\gamma,x,y)[\eta,\eta]
	&\gtrsim |x-y|^{q-2} \int_0^1 \pars{ |\dot \eta_s|^2 - C|x-y| |\dot \eta_s| |\eta_s| - C|x-y|^2 |\eta_s|^2}\;ds.
\end{align*}
Young's and Poincar\'e's inequalities give, for a larger value of $C$,
\[
	D^2_\gamma I(\gamma,x,y)[\eta,\eta] \gtrsim |x-y|^{q-2} \pars{ 1 - C|x-y|^2} \int_0^1 |\dot \eta_s|^2\;ds.
\]
Set $r_0 := \frac{1}{\sqrt{2 C}}$. Then, if $(x,y) \in \Delta_{r_0}$,
\[
	D^2_\gamma I(\gamma,x,y)[\eta,\eta] \gtrsim |x-y|^{q-2} \nor{\eta}{W_0^{1,2}}^2.
\]
As a consequence, for $(x,y) \in \Delta_{r_0}$ with $x \ne y$, $I(\cdot,x,y)$ has a unique minimizer $\gamma = \gamma(x,y)$. Also, for any $\eta \in W_0^{1,2}$,
\begin{align*}
	\nor{ D^2_\gamma I(\gamma(x,y),x,y)[\cdot,\eta] }{\pars{ W_0^{1,2}}^*}
	\gtrsim |x-y|^{q-2} \nor{\eta}{W_0^{1,2}}.
\end{align*}
It follows from the implicit function theorem that $(x,y) \mapsto \gamma(x,y)$ is $C^1$, and, therefore, $L(x,y) = I(\gamma(x,y),x,y)$ is $C^1$ in $\Delta_{r_0}\backslash \Delta_0$.
	
Part \eqref{L:Lcoercive} implies that $L$ is differentiable on $\Delta_0$ with $D_x L = D_y L = 0$. The proportionality constant $C = C(R)$ from part \eqref{L:LLipschitz} satisfies $\lim_{R \to 0} C(R) = 0$, and we conclude that $L$ is $C^1$ in all of $\Delta_{r_0}$.

\eqref{L:Lsemiconcave} Fix a minimizer $\gamma \in \mcl A(x,y)$, let $h,k \in \RR^d$, and set $\eta_s := h + s(k-h)$. Then $\gamma \pm \eta \in \mcl A(x\pm h,y\pm k)$, and, for all $s \in [0,1]$,
\begin{equation}
	|\eta_s| + |\dot \eta_s| \lesssim |h| + |k|. \label{E:etabounds}
\end{equation}
	
It follows from \eqref{A:Hbounds} that
\begin{align*}
	\int_0^1 H^*(-\dot \gamma_s - \dot \eta_s, \gamma_s + \eta_s)\;ds + \int_0^t H^*(-\dot \gamma_s + \dot \eta_s, \gamma_s - \eta_s)\;ds - 2 \int_0^1 H^*(-\dot \gamma_s, \gamma_s)\;ds \lesssim \I + \II + \III,
\end{align*}
where
\begin{gather*}
	\I := \int_0^1 (|\dot \gamma_s +\dot \eta_s|^{q-2})|\dot \eta_s|^2\;ds, \quad \II := \int_0^1(|\dot \gamma_s|^{q-1} + |\dot \eta_s|^{q-1})|\dot \eta_s| |\eta_s|\;ds,\\
	\text{and} \quad \III := \int_0^1 (|\dot \gamma_s|^q + |\dot \eta_s|^q) |\eta_s|^2\;ds.
\end{gather*}
Part \eqref{L:Lcoercive} yields
\begin{align*}
	\II \lesssim R^{q-1}(|h|^2 + |k|^2) + o(|h|^2 + |k|^2) \quad \text{and} \quad
	\III \lesssim R^q(|h|^2 + |k|^2) + o(|h|^2 + |k|^2).
\end{align*}
When $q \ge 2$,
\begin{align*}
	\I &\lesssim R^{q-2} (|h|^2 + |k|^2) + o(|h|^2 + |k|^2),
\end{align*}
while for $q < 2$, we use the lower bound on $|\dot \gamma|$. For sufficiently small $h$ and $k$, $|\dot \gamma_s + \dot \eta_s| \ge |\dot \gamma_s| - |\dot \eta_s| \gtrsim \frac{1}{R}$. Therefore
\begin{align*}
	\I &\lesssim R^{2-q} (|h|^2 + |k|^2) + o(|h|^2 + |k|^2).
\end{align*}
Combining the estimates for $\I$, $\II$ and $\III$ in either case finishes the proof.
\end{proofofal}

\subsection{The comparison principle}

Lemma \ref{L:distancefunction}\eqref{L:Lsmoothstrip} is used to prove the following.

\begin{proposition}\label{P:comparison}
	Let $u$ and $v$ be respectively a bounded sub- and super-solution of \eqref{E:appeq}. Then, for all $t \ge 0$,
	\[
		\sup_{x \in \RR^d} \pars{ u(x,t) - v(x,t)} \le \sup_{x \in \RR^d} \pars{ u(x,0) - v(x,0)}.
	\]
\end{proposition}

If $q \le 2$, then $H$ is $C^2$, and Perron's method can be used in conjunction with the comparison principle to construct the unique solution of \eqref{E:appeq}, as shown by the author in \cite{Se2}.

The argument below is new even in the classical viscosity theory. We are able to avoid the usual strategy of subtracting penalizations of the form $|x|^2 + |y|^2$ to handle the unboundedness of $\RR^d$.

\newproof{proofofap}{Proof of Proposition \ref{P:comparison}}

\begin{proofofap}
	We argue by contradiction and assume, as in the proof of the classical comparison principle, that there exist $T > 0$, $t_0 \in (0,T]$, sufficiently small $\mu > 0$, and sufficiently large $\lambda > 0$ such that the function
	\[
		t \mapsto \sup_{(x,y) \in \RR^d\times \RR^d} \pars{ u(x,t) - v(y,t) - \lambda^{q-1} L(x,y)} - \mu t
	\]
	achieves it maximum on $[0,T]$ at $t_0$.
	
	For $\lambda > 0$, define
	\[
		\Phi_\lambda(x,y,s,t) := \pars{ \frac{\lambda}{1 + \lambda(W_t - W_s) } }^{q-1} L(x,y).
	\]
	Then Lemma \ref{L:distancefunction}\eqref{L:Lequations} and \eqref{L:Lsmoothstrip} yield that, whenever $(x,y) \in \Delta_{r_0}$ and $\omega_W(|s-t|) < \frac{1}{2\lambda}$,
	\[
		(x,s) \mapsto \Phi_\lambda(x,y,s,t) \quad \text{and} \quad	(y,t) \mapsto - \Phi_\lambda(x,y,s,t)
	\]
	are $C^1$ in respectively $x$ and $y$, and solve \eqref{E:smoothtestfunc}.
		
	Let $M_0 := \max \pars{ \nor{u}{\RR^d \times [0,T]}, \nor{v}{\RR^d \times [0,T]} }$, choose
	\begin{equation}
		\lambda > \frac{3}{2} \pars{ \frac{2M_0}{c_0 r_0^q} }^{\frac{1}{q-1}}, \label{E:biglambda}
	\end{equation}
	and, for $\theta > 0$, consider the auxiliary function
	\begin{equation}
		(s,t) \mapsto \sup_{(x,y) \in \RR^d \times \RR^d} \pars{ u(x,s) - v(y,t) - \Phi_\lambda(x,y,s,t) } - \frac{|s-t|^2}{2\theta} - \mu  \frac{s+t}{2}, \label{E:timeaux}
	\end{equation}
	which attains a maximum at some $(s_\theta,t_\theta) \in S_\lambda := \left\{ (s,t) \in [0,T]^2 : \omega_W(|s-t|) \le (2\lambda)\nv \right\}$. Then $\lim_{\theta \to 0} \frac{|s_\theta - t_\theta|^2}{2\theta} = 0$ and
	\begin{gather*}
		\lim_{\theta \to 0} \sup_{(x,y) \in \RR^d \times \RR^d} \pars{ u(x,s_\theta) - v(y,t_\theta) - \Phi_\lambda(x,y,s_\theta,t_\theta) } - \frac{|s_\theta-t_\theta|^2}{2\theta} - \mu  \frac{s_\theta+t_\theta}{2}\\
		= \max_{t \in [0,T] } \sup_{(x,y) \in \RR^d\times \RR^d} \pars{ u(x,t) - v(y,t) - \lambda^{q-1} L(x,y)} - \mu t\\
		= \sup_{(x,y) \in \RR^d\times \RR^d} \pars{ u(x,t_0) - v(y,t_0) - \lambda^{q-1} L(x,y)} - \mu t_0.
	\end{gather*}
	Therefore, as $\theta \to 0$, $(s_\theta,t_\theta) \to (t_0,t_0)$, and so, for sufficiently small $\theta$, we have $s_\theta > 0$, $t_\theta > 0$, and $\omega_W(|s_\theta - t_\theta|) < \frac{1}{2\lambda}$.
	
	We show that, for each fixed $y \in \RR^d$,
	\[
		s \mapsto \sup_{x \in \RR^d} \pars{ u(x,s) - v(y,t_\theta) - \Phi_\lambda(x, y,s,t_\theta)}
	\]
	is nonincreasing in the interval $[a,b] := \left\{ s \in [0,T] : \omega_W(|s - t_\theta|) < (2\lambda)\nv \right\}$, which yields the same for
	\begin{equation}
		s \mapsto \sup_{(x,y) \in \RR^d \times \RR^d} \pars{ u(x,s) - v(y,t_\theta) - \Phi_\lambda(x,y,s,t_\theta) }. \label{E:snonincreasing}
	\end{equation}
	Arguing by contradiction, we assume that, for some $\hat y \in \RR^d$ and sufficiently small $\alpha > 0$, the map $s \mapsto \sup_{x \in \RR^d} \pars{ u(x,s) - \Phi_\lambda(x,\hat y,s,t_\theta) } - \alpha s$ attains a maximum at some $\hat s \in (a,b]$. Since $\omega_W(|\hat s-t_\theta|) < \frac{1}{2\lambda}$, we have, for all $x \in \RR^d$,
	\[
		c_0 \pars{ \frac{2}{3} \lambda}^{q-1} |x- \hat y|^q \le \Phi_\lambda(x,\hat y,\hat s ,t_\theta).
	\]
	Therefore, for some $\hat x \in \RR^d$, the map $(x,s) \mapsto u(x,s) - \Phi_\lambda(x, \hat y, s, t_\theta) - \alpha s$ attains a maximum at $(\hat x, \hat s)$. 
	
	It follows from \eqref{E:biglambda} and
	\[
		c_0 \pars{ \frac{2}{3} \lambda}^{q-1} |\hat x- \hat y|^q \le u(\hat x, \hat s) - u(\hat y, \hat s) \le 2M_0
	\]
	that $(\hat x, \oline y) \in \Delta_{r_0}$. In view of Lemma \ref{L:distancefunction}\eqref{L:Lsmoothstrip}, we can apply Definition \ref{D:solution} to obtain the contradiction $\alpha \le 0$.
	
	A similar argument for the super-solution $v$ yields that
	\begin{equation}
		t \mapsto \inf_{(x,y) \in \RR^d} \pars{ v(y,t) - u(x,s_\theta) + \Phi_\lambda( x,y,s_\theta,t) } \label{E:tnondecreasing}
	\end{equation}
	is nondecreasing on $[c,d] := \left\{ t \in [0,T] : \omega_W(|s_\theta - t|) < (2\lambda)\nv \right\}$.
	
	We now return to the maximum point $(s_\theta,t_\theta)$ of \eqref{E:timeaux}. The map
	\[
		s \mapsto \sup_{(x,y) \in \RR^d} \pars{ u(x,s) - v(y,t_\theta) - \Phi_\lambda(x,y,s,t_\theta)} - \frac{|s-t_\theta|^2}{2\theta} - \mu \frac{s}{2}
	\]
	attains a maximum at $s_\theta$, and, since $\omega_W(|s_\theta - t_\theta| ) < \frac{1}{2\lambda}$, we have $s_\theta > a$. Because \eqref{E:snonincreasing} is nonincreasing, we have $\frac{\mu}{2} + \frac{s_\theta - t_\theta}{\theta} \le 0$. Similarly, the map
	\[
		t \mapsto \inf_{(x,y) \in \RR^d} \pars{ v(y,t) - u(x,s_\theta) + \Phi_\lambda(x,y,s_\theta,t) } + \frac{|s_\theta - t|^2}{2\theta} + \mu \frac{t}{2}
	\]
	attains a minimum at $t_\theta$ with $t_\theta > c$, so, because \eqref{E:tnondecreasing} is nondecreasing, we have $- \frac{\mu}{2} + \frac{s_\theta - t_\theta}{\theta} \ge 0$. We conclude that $\mu \le 0$, a contradiction, and the result follows.
\end{proofofap}

\subsection{An estimate for the solutions, and the proofs of \eqref{E:extension} and \eqref{E:equicontinuity}}
	We now present an estimate for the solutions of \eqref{E:appeq} with smooth signals. Below, $c_0$ and $C_0$ are as in \eqref{A:Hbounds}.
	
\begin{proposition}\label{P:uvestimate}
	For $u_0, v_0 \in BUC(\RR^d)$ and $\xi, \zeta \in C^1([0,\oo))$ with $\xi_0 = \zeta_0$, let $u$ be a sub-solution of
	\begin{equation*}
			u_t = H(Du, x)\dot \xi_t \quad \text{in } \RR^d \times (0,\oo), \qquad u(\cdot,0) = u_0 \quad \text{on } \RR^d,
	\end{equation*}
	and $v$ a super-solution of
	\begin{equation*}
			v_t = H(Dv, x)\dot \zeta_t \quad \text{in } \RR^d \times (0,\oo), \qquad v(\cdot,0) = v_0 \quad \text{on } \RR^d.
	\end{equation*}
	Then, for all $T > 0$ and $0 < \lambda < \pars{ \max_{0 \le t \le T} \pars{\xi_t - \zeta_t}_+ }\nv$,
	\begin{gather*}
		\sup_{(x,y,t) \in \RR^d \times \RR^d \times [0,T]} \pars{ u(x,t) - v(y,t) - C_0 \pars{ \frac{ \lambda}{1 - \lambda(\xi_t - \zeta_t)} }^{q-1} |x-y|^q} \\
		\le \sup_{(x,y) \in \RR^d \times \RR^d} \pars{ u_0(x) - v_0(y) - c_0 \lambda^{q-1} |x-y|^q}.
	\end{gather*}
\end{proposition}

We first use Proposition \ref{P:uvestimate} to prove \eqref{E:extension} and \eqref{E:equicontinuity}.

\newproof{proofofae}{Proof of \eqref{E:extension}}

\begin{proofofae}
	Assume first that $W^1, W^2 \in C^1([0,\oo),\RR)$. Proposition \ref{P:uvestimate} gives, for all $0 < \lambda < \nor{W^1 - W^2}{[0,t]}\nv$ and $x \in \RR^d$,
	\[
		|u^1(x,t) - u^2(x,t)| \le \nor{u_0^1 - u_0^2}{\oo} + \theta(\omega_{u_0^1,u_0^2},\lambda).
	\]
	Letting $\lambda \to \nor{W^1 - W^2}{[0,t]}\nv$ gives the result for smooth $W^1$ and $W^2$. A standard density argument yields the existence for solutions of \eqref{E:appeq}, and shows that \eqref{E:extension} holds for all continuous $W^1$ and $W^2$.
	\end{proofofae}
	
\newproof{proofofaee}{Proof of \eqref{E:equicontinuity}}

\begin{proofofaee}
	In view of \eqref{E:extension}, we may assume $W \in C^1([0,\oo), \RR)$. We apply Proposition \ref{P:uvestimate} with $u_0^1 = u_0^2 = u_0$ and $\xi = \zeta = W$, and find, for any $\lambda > 0$, $x,y \in \RR^d$, and $t \ge 0$,
	\begin{align*}
		u(x,t) &- u(y,t) \le C_0\lambda^{q-1} |x-y|^q + \sup_{(x,y) \in \RR^d \times \RR^d} \pars{ u_0(x) - u_0(y) - c_0 \lambda^{q-1} |x-y|^q}\\
		&\le C_0 \lambda^{q-1} |x-y|^q + \sup_{(x,y) \in \RR^d} \pars{ \omega_{u_0}(|x-y|) - c_0 \lambda^{q-1} |x-y|^q} \le C_0 \lambda^{q-1} |x-y|^q + \theta(\omega_{u_0}, \lambda).
	\end{align*}
	Taking the infimum over $\lambda$ finishes the proof.
\end{proofofaee}

\newproof{proofofapp}{Proof of Proposition \ref{P:uvestimate}}

\begin{proofofapp}
	Classical viscosity solution arguments show that $z(x,y,t) := u(x,t) - v(y,t)$ is a sub-solution of
	\begin{equation}
		z_t = H(D_x z,x)\dot \xi_t - H(-D_y z, y) \dot \zeta_t \quad \text{in } \RR^d \times \RR^d \times (0,\oo). \label{E:mixeddoubled}
	\end{equation}
	
	For $0 < \lambda < \pars{ \max_{0 \le t \le T} (\xi_t - \zeta_t)_+}\nv$, define
	\[
		\Phi_\lambda(x,y,t) := \pars{ \frac{\lambda}{1 - \lambda(\xi_t - \zeta_t)} }^{q-1} L(x,y).
	\]
	It is immediate that $\Phi_\lambda(x,y,0) = \lambda^{q-1} L(x,y)$, and, in view of Lemma \ref{L:distancefunction}\eqref{L:Lequations}, $\Phi$ solves \eqref{E:mixeddoubled} whenever $L$ is differentiable at $(x,y)$.
	
	Next, for $0 < \beta < 1$ and $\mu > 0$, define
	\[
		\Psi(x,y,t) := u(x,t) - v(y,t) - \Phi_\lambda(x,y,t) - \frac{\beta}{2}(|x|^2 + |y|^2) - \mu t.
	\]
	The comparison principle yields that $u$ and $v$ are bounded, and, therefore, $\Psi$ attains a maximum on $\RR^d \times \RR^d \times [0,T]$ at some $(\hat x, \hat y, \hat t)$ that depends on $\beta$ and $\mu$. Assume for the sake of contradiction that $\hat t > 0$. 
	
	Rearranging terms in the inequality $\Psi(0,0,\hat t) \le \Psi(\hat x, \hat y, \hat t)$ gives
	\[
		\frac{\beta}{2}(|\hat x|^2 + |\hat y|^2) \le u(\hat x, \hat t) - v(\hat y, \hat t) - (u(\hat 0, \hat t) - v(\hat 0,\hat t)) \le 2 (\nor{u}{\RR^d \times [0,T]} + \nor{v}{\RR^d \times [0,T]}).
	\]
	The inequality $\Psi(\hat y, \hat y, \hat t) \le \Psi(\hat x, \hat y, \hat t)$ and Lemma \ref{L:distancefunction}\eqref{L:Lcoercive} yield
	\[
		c_0 \pars{ \frac{\lambda}{1 - \lambda(\xi_{\hat t} - \zeta_{\hat t}) } }^{q-1} |\hat x - \hat y|^q \le u(\hat x, \hat t) - u(\hat y, \hat t) + \frac{\beta}{2} (|\hat y|^2 - |\hat x|^2).
	\]
	It follows that, for some $R >0$ depending on $\lambda$, $u$, $v$, $\xi$, and $\zeta$, $\pars{ |\hat x|^2 + |\hat y|^2 }^{1/2} \le R \beta^{-1/2}$ and $(\hat x, \hat y) \in \Delta_R$. 
	
	The definition of $\Phi_\lambda$ requires $\lambda(\xi_t - \zeta_t) < 1$. Therefore, since there is no restriction on the size of $\xi - \zeta$, we cannot expect $\lambda$ to be large enough to force $(\hat x, \hat y) \in \Delta_{r_0}$. Instead, we double variables once more and take advantage of the semiconcavity of $L$. 
	
	For $0 < \delta < 1$, set
	\begin{gather*}
		\Psi_\delta(x,y,z,w,t) := u(x,t) - v(y,t) -\frac{1}{2\delta}(|x-z|^2 + |y-w|^2) \\
		- \Phi_\lambda(z,w,t) - \frac{\beta}{2}(|z|^2 + |w|^2) - \mu t - \frac{1}{2}\pars{ |x - \hat x|^2 + |y - \hat y|^2 + |t - \hat t|^2}
	\end{gather*}
	and $\Omega_{R,\beta} := \Delta_R \cap B_{R \beta^{-1/2}} \subset \RR^d \times \RR^d$, and assume that the maximum of $\Psi_\delta$ on $\Omega_{R,\beta} \times \Omega_{R,\beta} \times [0,T]$ is attained at $(x_\delta,y_\delta,z_\delta,w_\delta,t_\delta)$.
		
	Lemma \ref{L:distancefunction}\eqref{L:LLipschitz} gives $|D_z \Phi_\lambda| + |D_w \Phi_\lambda| + \beta (|z| + |w|) \lesssim_R 1$ on $ \Omega_{R,\beta} \times \Omega_{R,\beta} \times [0,T]$. Rearranging terms in the inequality $\Psi_\delta(x_\delta, y_\delta, x_\delta, y_\delta, t_\delta) \le \Psi_\delta(x_\delta,y_\delta,z_\delta,w_\delta,t_\delta)$ yields
	\begin{align*}
		\frac{1}{2\delta} \pars{ |x_\delta - z_\delta|^2 + |y_\delta - w_\delta|^2}
		&\le \Phi_\lambda(x_\delta, y_\delta, t_\delta) - \Phi_\lambda(z_\delta,w_\delta,t_\delta) + \frac{\beta}{2}(|x_\delta|^2 + |y_\delta|^2 - |z_\delta|^2 - |w_\delta|^2) \\
		&\lesssim_R \pars{ |x_\delta - z_\delta| + |y_\delta - w_\delta|},
	\end{align*}
	and, hence, $|x_\delta - z_\delta| + |y_\delta - w_\delta| \lesssim_R \delta$.
	
	Since $(\hat x, \hat y, \hat x, \hat y, \hat t) \in \Omega_{R,\beta} \times \Omega_{R,\beta} \times [0,T]$ and $\Psi_\delta(\hat x, \hat y, \hat x, \hat y, \hat t) = \Psi(\hat x, \hat y, \hat t)$,
	\begin{align*}
		\Psi(\hat x, \hat y, \hat t)
		&= u(\hat x, \hat t) - v(\hat y, \hat t) - \Phi_\lambda(\hat x, \hat y, \hat t) - \frac{\beta}{2} (|\hat x|^2 + |\hat y|^2) - \mu \hat t\\
		&\le u(x_\delta, t_\delta) - v(y_\delta, t_\delta) - \frac12 |x_\delta - \hat x|^2 - \frac12 |y_\delta - \hat y|^2 
		-\frac{1}{2\delta}(|x_\delta - z_\delta|^2 + |y_\delta - w_\delta|^2)\\
		& - \Phi_\lambda(z_\delta, w_\delta, t_\delta) 
		- \frac{\beta}{2} (|z_\delta|^2 + |w_\delta|^2) - \mu t_\delta - \frac12 |t_\delta - \hat t|^2 \\ 
		&\le \Psi(x_\delta, y_\delta, t_\delta) + \Phi_\lambda(x_\delta,y_\delta,t_\delta) - \Phi_\lambda(z_\delta,w_\delta,t_\delta) +\frac{\beta}{2} (|x_\delta|^2 + |y_\delta|^2 - |z_\delta|^2 - |w_\delta|^2) \\
		&- \frac 12 \pars{ |x_\delta - \hat x|^2 - |y_\delta - \hat y|^2 - |t_\delta - \hat t|^2 }.
	\end{align*}
	Rearranging terms and using $\Psi(x_\delta,y_\delta,t_\delta) \le \Psi(\hat x, \hat y, \hat t)$, we see that
	\[
		|x_\delta - \hat x|^2 + |y_\delta - \hat y|^2 + |t_\delta - \hat t|^2 \lesssim_R \delta.
	\]
	Therefore, for sufficiently small $\delta$, $(x_\delta,y_\delta,z_\delta,w_\delta, t_\delta)$ is a local interior maximum point of $\Psi_\delta$ in $\Omega_{R,\beta} \times \Omega_{R,\beta} \times (0,T]$.
				
	Since
	\begin{align*}
		(x,y,t) &\mapsto u(x,t) - v(y,t) - \frac{1}{2\delta} \pars{ |x- z_\delta|^2 + |y-w_\delta|^2}\\
		& - \Phi_\lambda(z_\delta,w_\delta,t) - \mu t - \frac{1}{2} \pars{ |x-\hat x|^2 - |y - \hat y|^2 - |t - \hat t|^2}
	\end{align*}
	attains an interior maximum at $(x_\delta,y_\delta,t_\delta)$, the definition of viscosity solution for the doubled equation \eqref{E:mixeddoubled} yields
	\begin{align*}
		\mu + t_\delta - \hat t + \Phi_{\lambda,t}(z_\delta,w_\delta,t_\delta)
		\le H\pars{ \frac{x_\delta - z_\delta}{\delta} + x_\delta - \hat x, x_\delta} \dot \xi_{t_\delta}
		- H\pars{ - \frac{y_\delta - w_\delta}{\delta} - (y_\delta - \hat y), y_\delta} \dot \zeta_{t_\delta}.
	\end{align*}
	
	Next, $(z_\delta,w_\delta)$ is a minimum point of
	\begin{align*}
		(z,w) \mapsto \Phi_\lambda(z,w,t_\delta) + \frac{1}{2\delta} (|x_\delta - z|^2 + |y_\delta - w|^2) + \frac{\beta}{2} (|z|^2 + |w|^2).
	\end{align*}
	In view of Lemma \ref{L:distancefunction}\eqref{L:Lsemiconcave}, $\Phi_\lambda$ is differentiable at $(z_\delta,w_\delta)$, and so
	\[
		D_x \Phi_\lambda(z_\delta,w_\delta,t_\delta) = \frac{x_\delta - z_\delta}{\delta} - \beta z_\delta, \quad
		D_y \Phi_\lambda(z_\delta,w_\delta,t_\delta) = \frac{y_\delta - w_\delta}{\delta} - \beta w_\delta,
	\]
	and
	\[
		\Phi_{\lambda,t}(z_\delta,w_\delta,t_\delta)
		= H(D_x \Phi_\lambda(z_\delta,w_\delta,t_\delta), z_\delta) \dot \xi_{t_\delta} - H(-D_y \Phi_\lambda(z_\delta,w_\delta,t_\delta), w_\delta) \dot \zeta_{t_\delta}.
	\]
	
	It follows that
	\begin{align*}
		\mu + t_\delta - \hat t + \Phi_{\lambda,t}(z_\delta,w_\delta,t_\delta)
		&\le H\pars{ D_x \Phi_\lambda(z_\delta,w_\delta,t_\delta) + \beta z_\delta + x_\delta - \hat x, x_\delta} \dot \xi_{t_\delta}\\
		&- H\pars{ - D_y \Phi_\lambda(z_\delta, w_\delta,t_\delta) - \beta w_\delta - (y_\delta - \hat y), y_\delta} \dot \zeta_{t_\delta}.
	\end{align*}
	The bounds for $(\hat x, \hat y, \hat t)$ and $(x_\delta,y_\delta,z_\delta,w_\delta,t_\delta)$ and \eqref{A:Hbounds} yield
	\[
		\mu \lesssim_R (\beta^{1/2} + \delta^{1/2} + \delta)\pars{\nor{\dot \xi}{[0,T]} + \nor{\dot \zeta}{[0,T]}}.
	\]
	We obtain a contradiction for sufficiently small enough $\delta$ and $\beta$.
	
	Therefore, for all $\mu >0$ and $t \in [0,T]$,
	\begin{gather*}
		\lim_{\beta \to 0} \sup_{(x,y) \in \RR^d \times \RR^d} \pars{ u(x,t) - v(y,t) - \Phi_\lambda(x,y,t) - \frac{\beta}{2} (|x|^2 + |y|^2)}\\
		= \sup_{(x,y) \in \RR^d \times \RR^d} \pars{ u(x,t) - v(y,t) - \Phi_\lambda(x,y,t)}
		\le \sup_{(x,y) \in \RR^d \times \RR^d} \pars{ u_0(x) - v_0(y) - \lambda^{q-1}L(x,y)} + \mu t.
	\end{gather*}
	Letting $\mu \to 0$ and applying Lemma \ref{L:distancefunction}\eqref{L:Lcoercive} finishes the proof.
	\end{proofofapp}

\section*{Acknowledgements}
	I would like to thank my advisor, Professor Panagiotis E. Souganidis, for suggesting the topic of this paper, and Professor Pierre-Louis Lions for some helpful remarks.

	I am partially supported by the NSF grants NSF-DMS-1266383, NSF-DMS-1600129, and RTG-DMS-1246999.

\end{document}